\documentclass[10pt]{amsart}

\usepackage{amsmath,  amsfonts, amssymb, amsthm, bm, mathtools}
\usepackage[hidelinks]{hyperref}

\newcommand{\MM}{\mathbb{M}}
\newcommand{\FF}{\mathbb{F}}
\newcommand{\CC}{\mathbb{C}}
\newcommand{\TT}{\mathbb{T}}
\newcommand{\SSS}{\mathbb{S}}
\newcommand{\EE}{\mathbb{E}}
\newcommand{\ZZ}{\mathbb{Z}}

\newcommand{\mfrak}{\mathfrak{m}}
\newcommand{\pfrak}{\mathfrak{p}}
\newcommand{\qfrak}{\mathfrak{q}}
\newcommand{\vfrak}{\mathfrak{v}}
\newcommand{\gfrak}{\mathfrak{g}}
\newcommand{\pitilde}{\widetilde{\pi}}
\newcommand{\Hcal}{\mathcal{H}}
\newcommand{\Fcal}{\mathcal{F}}
\newcommand{\Gcal}{\mathcal{G}}
\newcommand{\Ecal}{\mathcal{E}}
\DeclareMathOperator{\Hol}{Hol}
\DeclareMathOperator{\Thol}{THol}
\DeclareMathOperator{\GL}{GL}
\DeclareMathOperator{\Mat}{M}
\DeclareMathOperator{\ev}{ev}
\newcommand{\hyp}{\mathcal{D}}
\newcommand{\GLA}{\GL_2(A)}
\newcommand{\lquo}[2]{\displaystyle{\raisebox{-.1em}{$#1$}\left\backslash \raisebox{.1em}{$#2$}\right.}}
\newcommand{\ec}{\mathfrak{e_c}}
\newcommand{\dd}{\bm{d}}

\newtheorem{theorem}{Theorem}[section]
\newtheorem{proposition}[theorem]{Proposition}
\newtheorem{lemma}[theorem]{Lemma}
\newtheorem{corollary}[theorem]{Corollary}
\newtheorem{problem}[theorem]{Problem}
\theoremstyle{remark}
\newtheorem{remark}[theorem]{Remark}

\theoremstyle{definition}
\newtheorem{definition}[theorem]{Definition}
\numberwithin{equation}{section}

\title{Vectorial Drinfeld modular forms over Tate algebras}
\author{F. Pellarin \& R. Perkins}
\email{federico.pellarin@univ-st-etienne.fr}
\email{rudolph.perkins@iwr.uni-heidelberg.de}
\address{Federico Pellarin. 
Institut Camille Jordan, Universit\'e Claude Bernard Lyon 1, 43 boulevard du 11 novembre 1918, 69622 Villeurbanne cedex, France}
\address{Rudolph B. Perkins. 
Institut Camille Jordan, Universit\"at Heidelberg, Im Neuenheimer Feld 368, 69120 Heidelberg, Germany}


\begin{document}

\begin{abstract}
In this text, we develop the theory of vectorial modular forms with values in Tate algebras
introduced by the first author, in a very special case (dimension two, for a very particular 
representation of $\Gamma:=\operatorname{GL}_2(\FF_q[\theta])$). Among several results that we prove here, we determine the complete structure
of the modules of these forms, we describe their specializations at roots of unity and their
connection with Drinfeld modular forms for congruence subgroups of $\Gamma$
and we prove that the modules generated by these forms are stable under the actions
of Hecke operators.
\end{abstract}

\maketitle

\tableofcontents

\section{Introduction}

It might be surprising at first sight, to read that 
the origin of the present paper, and of the functions of the title, in fact goes back to the 
article of Kaneko and Koike \cite{KAN&KOI}. The main idea is that the contiguity relations associated to the
{\em differential equations of hypergeometric type} introduced therein are formal
analogues of certain linear, homogeneous twisted Frobenius difference equations satisfied by 
{\em vectorial modular forms with values in Tate algebras.} 

The following very particular classical example can guide the reader in quest of analogies with the
classical world of complex valued elliptic and modular forms. 
Consider indeed the vectorial function $$z\mapsto\binom{\eta_1(z)}{\eta_2(z)},\quad \Re(z)>0$$ associating to $z$ in the complex upper-half plane, the fundamental quasi-periods $\eta_1(z),\eta_2(z)$ of the lattice $z\ZZ+\ZZ$
which is also a vectorial modular form for the group $\operatorname{SL}_2(\ZZ)$. 

The analogue of this map in the settings of first author's paper \cite{FPannals} has a deformation (in our terminology) into a {\em 
weak} vectorial modular form with values in the standard one dimensional affinoid algebra
(also called {\em Tate $\CC_\infty$-algebra})
$\TT=\widehat{\CC_\infty[t]}$, where $\CC_\infty$ is the complete, algebraically closed 
field $\widehat{\FF_q((\theta^{-1}))^{ac}}$ (completion of an algebraic closure of the local field 
$K_\infty:=\FF_q((\theta^{-1}))$) and where the completion of $\CC_\infty[t]$ is taken for the Gauss valuation, trivial over $\FF_q[t]$.

The above deformation is in fact deeply connected with a family of non-singular $2\times 2$ matrices 
$\Psi(z,t)$ with entries in $\TT$, for $z$ a modular parameter, which occur, for $z$ fixed,
as fundamental matrices of certain twisted Frobenius difference linear systems associated to
Anderson's $t$-motives associated to the lattice $Az+A$ with $A:=\FF_q[\theta]$, 
and described in \cite{FPbourbaki}. In the paper \cite{FPcrelles}, all the elements are given to track the above mentioned 
analogy.

In the classical theory of vectorial modular forms for $\operatorname{SL}_2(\ZZ)$ or for its 
congruence subgroups (the reference \cite{MAS} is perhaps the closest one to the scope and spirit of the present paper, but the literature is by far more vast), natural generalizations of Eisenstein series,
Poincar\'e series etc. occur. Similarly, we can easily construct such series in our framework,
but as far as we can see, no classical analogue of Anderson's matrix function $\Phi(z,t)$ has been observed. Since moreover,
these functions have been used in a crucial way in \cite{FPannals} to obtain certain new functional identities
between {\em zeta-values} in $\TT$ (see also the subsequent works \cite{APinv,APTR}), we think that
there are sufficiently many reasons to deepen the study of vectorial modular forms (abridged to 
VMF in all the following) with values in Tate algebra, in the direction suggested by the papers \cite{FPcrelles,FPannals}.

With this paper, we have tried to make the theory of VMF for the representation $\rho_t^*$.
Let $\Omega:=\CC_\infty\setminus K_\infty$ be the Drinfeld upper-half plane as defined in \cite{EGinv}. Explicitly, a {\em VMF of weight $k$ and 
type $m$ for $\rho_t^*$} is a vector holomorphic function
$f:\Omega\rightarrow\TT^2$, in the sense of \cite{FPRP} (see also \S \ref{tvaluedholom}), satisfying the following collection of 
functional equations:
$$f(\gamma(z))=(cz+d)^k\det(\gamma)^{-m-1}\begin{pmatrix} d(t) & -c(t) \\ -b(t) & a(t)\end{pmatrix}\cdot f(z),\quad \gamma=\begin{pmatrix} a & b \\ c & d\end{pmatrix}\in\Gamma,$$
where $\Gamma$ is the Drinfeld modular group $\operatorname{GL}_2(A)$, acting 
over $\Omega$ by homographies in the usual way; the vector function $f$ must also satisfy
a growth condition at infinity (see Definition \ref{VMFdef}). We are going, for fixed $k,m$, to
study the structure of the $\TT$-module of these vector functions.

The paper starts with a review of the basic tools we need to use; the Drinfeld upper-half plane, Tate 
algebras, uniformizers etc. (see \S \ref{uniformizers}). The Section \ref{valuedvectorialmodularforms} starts with the essential definitions of weak 
vectorial modular form and vectorial modular form (in our setting)
and provides the first examples: vectorial Eisenstein series and 
the so-called Anderson generating functions, which however are not,
properly speaking, VMF, but VMF${}^!$, that is, {\em weak} vectorial modular forms
(indeed, the growth condition at infinity of Definition \ref{VMFdef} fails).
After this, we immediately state and prove our structure result (Theorem \ref{structurethm}) which is a refinement of \cite[Proposition 19]{FPannals}.
This result is then applied to the computation of a $\tau$-difference
equation satisfied by the Eisenstein series of weight one (in \S \ref{differenceE1}) and several properties related to evaluation at
$t=\theta^{q^k}$ with $k\geq 0$ an integer, such as {\em Drinfeld quasi-modular forms} as in \cite{BvPfimrn} (in \S \ref{qthpowersevalsect}), and 
Petrov's special families of Drinfeld modular forms with {\em $A$-expansion}.
Other topics explored in this \S \ref{valuedvectorialmodularforms} are: an explicit computation of $A$-expansions of our vectorial Eisenstein series (in Theorem \ref{eisAexp}),
and Ramanujan-Serre derivatives in \S \ref{Ramanujan-Serre}.

In \S \ref{InterpolationofDrinfeldmodularforms}, we focus on the 
intricate interplay between VMF and, via specialization at roots of unity,
Drinfeld modular forms for congruence subgroups of $\Gamma$ with 
prime level. It is precisely at this point that the reader will realize how subtle 
is the condition of regularity at the cusp infinity of Definition \ref{VMFdef}.
Indeed, this condition is the weakest possible, ensuring that, 
given a VMF, the evaluation at roots of unity of its coordinate functions
are Drinfeld modular forms for the group $\Gamma_0(P)$
for some $P$, and with character. 
Interesting specialization
properties are known; for instance, specializing the weight one vectorial Eisenstein series at roots of a prime $P$ of degree $d$
allows to span a canonical $2d$-dimensional sub-vector space of the space
of modular forms for the group $\Gamma_1(P)$ which are also  modular
for $\Gamma_0(P)$ with a character $$\Gamma_0(P)/\Gamma_1(P)\rightarrow\FF_{q^d}^\times,$$ to only mention  one result of this section.

The main result of this Section is thus 
Proposition \ref{eigenevalprop}, immediately yielding the results
of \S \ref{exEWCwt1} on specializations of the vectorial Eisenstein series
of weight $1$ at roots of unity, and various other results also including 
powers of primes levels. In particular, in Theorem \ref{reginftythm} the reader will find various equivalent characterizations of the growth condition 
at infinity, in terms of specializations of VMF${}^!$ at roots of unity. 

Finally, in \S \ref{Heckeoperators} we use various properties obtained 
in \S \ref{valuedvectorialmodularforms} and \S \ref{InterpolationofDrinfeldmodularforms} to analyze the action of Hecke operators on our modules of VMF. Indeed, it is not at all trivial that our condition of regularity at infinity is preserved under action of Hecke operators, but this is so (see \S \ref{regularityhecke}). Thanks to this, examples of
vectorial Hecke eigenforms are given in \S \ref{eigenformsFirstexamples},
notably vectorial Eisenstein series, as detailed in Proposition \ref{EisHeckEFs}.

As a final remark of this introduction, we shall say something about further possible developments of the theory. For example, we should consider the irreducible representation
$$\rho_s:=\rho_{t_1}^*\otimes\cdots\otimes\rho_{t_s}^*$$
for independent variables $t_1,\ldots,t_s$ and the associated VMF
$$\Omega\rightarrow\TT_s^{2^s},$$ with $\TT_s=\widehat{\CC_\infty[t_1,\ldots,t_s]}$.
It would be nice to introduce a suitable condition of regularity at
infinity and generalize our Definition \ref{VMFdef} in such a way that
the results of the present paper could be extended to this natural 
setting. This would be very interesting for the theory of 
$L$-values as in \cite{FPannals}. In particular, we address the following problem, in which it is understood a notion of regularity at infinity which is compatible with the various specializations at $t_1,\ldots,t_s$ roots of unity and good behavior of Hecke operators which is unknown at the moment.
\begin{problem} Let $s$ be congruent to $1$ modulo $q-1$.
Show that the $\TT_s$-module of $VMF$ of weight $1$ and type $0$ for the representation $\rho_s$ is free of rank one, 
generated by the vectorial Eisenstein series of weight one.
\end{problem}

\section{Basic tools}\label{uniformizers}

Let $R$ be a ring. In all the following, we denote by $\operatorname{Mat}_{n\times m}(R)$ the set of 
matrices with $n$ rows and $m$ columns with entries in $R$. We more simply write
$R^l$ for $\operatorname{Mat}_{n\times 1}(R)$.
We also denote by $R^\times$
the group of the invertible elements of $R$. 

\subsection{The Drinfeld upper-half space}
Define the set $\Omega := \CC_\infty \setminus K_\infty$, and equip it with the \emph{imaginary part}
map
\[|z|_\Im := \inf_{\kappa \in K_\infty} |z-\kappa|.\]
We give $\Omega$ the structure of a connected rigid-analytic space by equipping it with the affinoid open cover $\Omega = \cup_{n \geq 0} \Omega_n$, where, for each $n \geq 0$,
\[\Omega_n := \{ a \in \Omega : |z| \leq q^n \text{ and } |z|_\Im \geq q^{-n} \}.\]
We refer to $\Omega$ as the \emph{Drinfeld upper-half space}, in analogy with the setting over the classical complex numbers. Indeed, here $\CC_\infty, K_\infty$ and $|\cdot|_{\Im}$ play the role of the complex numbers, real numbers and the classical imaginary part of a complex number, respectively. 

\subsubsection{Action of $\GL_2(A)$ on $\Omega$}
The group $\GL_2(A)$ acts on $\Omega$ via fractional linear transformations $$\left(\begin{matrix} a & b \\ c & d  \end{matrix} \right)z = \frac{az+b}{cz + d}$$ in a way that is compatible with the rigid analytic structure. In other words, the quotient space $\lquo{\GLA}{\Omega}$ inherits the structure of a rigid analytic space.

\subsection{Tate algebra}
\begin{definition}
The \emph{Tate algebra} $\TT$ (standard of dimension one) is the completion of the polynomial ring $\CC_\infty[t]$ equipped with the \emph{Gauss norm}, defined by $\left\|\sum_{i \geq 0} f_i t^i \right\| := \max_{i \geq 0} |f_i|$.
\end{definition}

We can identify $\TT$ with the ring of formal series $f=\sum_{i\geq 0}f_it^i$, with $f_i\in\CC_\infty$
and $f_i\rightarrow0$, so that $\|f\|=\sup_i|f_i|=\max_i|f_i|$. It is isomorphic to the algebra of rigid analytic functions in the variable $z$ over the 
disk $\{z\in\CC_\infty;|z|\leq 1\}$ and contains as a subring the ring $\EE$ of entire functions $\CC_\infty\rightarrow\CC_\infty$. The isomorphism is defined by sending the indeterminate $t$
to the variable $z$.

\subsubsection{Anderson twists}
The space $\TT$ comes equipped with a continuous action of the twisted polynomial ring 
$\CC_\infty\{\tau\}$ determined by the continuous $\FF_q[t]$-linear algebra action $\tau$ given by
\[\tau\left( \sum_{i \geq 0} f_i t^i \right) := \sum_{i \geq 0} f_i^q t^i. \] It is well known that 
$$\TT^{\tau=1}:=\{f\in\TT;\tau(f)=f\}=\FF_q[t],$$ see e. g. Papanikolas' \cite{PAP}.

Viewing $\CC_\infty$ embedded in $\TT$ via $z \mapsto z \cdot 1$, we observe that the action of $\tau$ is an extension of the $q$-power Frobenius of $\CC_\infty$.

\subsubsection{$\TT^{\tau=1}$-valued representations}
We define the $\FF_q$-algebra map $\chi_t: A \rightarrow \FF_q[t] \subset \TT$ via $\theta \mapsto t$.
Occasionally, we shall also write $a(t)$ in place of $\chi_t(a)$.
We observe that the invariant elements $\TT^\tau$ under the action of $\tau$ are exactly those of the ring $\FF_q[t]$, and thus we consider $\chi_t$ as an extension of the notion of Dirichlet character with values in $\TT$, see \cite{APTR} for more on this point of view.

The character $\chi_t$ gives rise to the representation $\rho_t : \GL_2(A) \rightarrow \GL_2(\FF_q[t])$ given by $(a_{ij}) \mapsto (\chi_t(a_{ij}))$, and we write $\rho^*_t : \GL_2(A) \rightarrow \GL_2(\FF_q[t])$ for $\rho_t$ followed by taking inverse and transpose; i.e. $\rho^*_t(a_{ij}) := \rho_t(a_{ij})^{-tr}$, for all $(a_{ij}) \in \GL_2(A)$.
Explicitly:
$$\rho_t^*\left(\begin{matrix} a & b \\ c & d\end{matrix}\right):=\delta^{-1}\left(\begin{matrix} \chi_t(d) & -\chi_t(c) \\ -\chi_t(b) & \chi_t(a)\end{matrix}\right),$$ where $\delta := ad-bc \in \FF_q^\times$ is the determinant of $(\begin{smallmatrix} a & b \\ c & d \end{smallmatrix})$.

Together with the representations $\rho_t$ and $\rho_t^*$, we also need a symbol to designate the trivial 
representation 
$$\boldsymbol{1}:\GLA\rightarrow\{1\}$$
which sends any $\gamma\in\GLA$ to $1\in\TT$.

\subsubsection{$\TT$-valued rigid analytic functions}\label{tvaluedholom}
Let $$u(z) := \frac{1}{\pitilde} \sum_{a \in A} \frac{1}{z-a}$$ be Goss' uniformizer for the cusp at infinity on $\Omega$, where 
\begin{align} \label{pitildedef}
\pitilde := -\iota_\theta \theta \prod_{i \geq 1} \left(1 -\frac{\theta}{\theta^{q^i}}\right)^{-1} \end{align}
is the fundamental period of the Carlitz module, and $\iota_\theta^{q-1} = -\theta$ is a fixed element of Carlitz $\theta$-torsion.
We refer to \cite{EGinv} for the basic theory of Drinfeld modular forms.
In particular, we will adopt the same notations and terminology of ibid. The have the quasi-modular $E$ form of weight $2$ type $1$ and depth $1$,
the modular form $h$ of weight $q+1$ and type $1$ (a Poincar\'e series), and the modular form
$g$ of weight $q-1$ and type $0$ (an Eisenstein series). The $\CC_\infty$-algebra
of Drinfeld modular forms is equal to $\CC_\infty[g,h]$ and isomorphic to 
a polynomial ring in two indeterminates with coefficients in $\CC_\infty$.
More precisely, the functions $E,g,h$ have the following properties.

The function $g$ is proportional to an Eisenstein series
(see \cite[Section 2]{GossCM} and \cite[Section (6.4) p. 683]{EGinv}):
$$g(z)=\widetilde{\pi}^{1-q}\sideset{}{'}\sum_{a,b\in A}(az+b)^{1-q}.$$
We recall that $g$ modular of weight $q-1$ means that
for all $\gamma=\begin{pmatrix}a&b\\ c&d\end{pmatrix}\in\Gamma$ and $z\in\Omega$:
$$g(\gamma(z))=(cz+d)^{q-1}g(z).$$
Moreover, there is a locally convergent $u$-expansion whose first terms are: 
\begin{equation}g(z)=1-[1]v-[1]v^{q^2-q+1}+\cdots,\label{explicitg}\end{equation}
where $[1]$ denotes the polynomial $\theta^q-\theta$ and $v=u^{q-1}$ (that is, convergent for $z\in\Omega$ such that $|u|$ is small enough, with $u=u(z)$).

As for the function $E$, it can be defined by the
conditionally convergent series \cite[p. 686]{EGinv}:
$$E(z)=\widetilde{\pi}^{-1}\sum_{a\in A^+}\sum_{b\in A}\frac{a}{az+b},$$
where $A^+$ denotes the subset of monic polynomials of $A$.
It is easy to show that for $\gamma\in\Gamma$ as above,
\begin{equation}E(\gamma(z))=(cz+d)^2\det(\gamma)^{-1}\left(E(z)-\frac{c}{\widetilde 
{\pi}(cz+d)}\right),\label{formE}\end{equation}
with $u$-expansion
\begin{equation}E(z)=u(1+v^{q-1}+\cdots).\label{explicitE}\end{equation}
For the function $h$, finally,
we have, by using 
a variant of Ramanujan's derivative of modular forms:
$$h(z)=\partial g(z)=\widetilde{\pi}^{-1}\frac{dg(z)}{dz}-E(z)g(z),$$ as in \cite[Theorem (9.1) p. 687]{EGinv}.
We verify that for $\gamma\in\Gamma$ as above,
$$h(\gamma(z))=(cz+d)^{q+1}\det(\gamma)^{-1}h(z)$$
and there is a $u$-expansion defined over $A$:
\begin{equation}h(z)=-u(1+v^{q-1}+\cdots).\label{explicith}\end{equation}

We have, for $z\in\Omega$ such that $|u(z)|$ is small enough, series expansions
with coefficients in $A$, locally convergent at $0$
(with $u=u(z)$ and  $v=u^{q-1}$):
\begin{eqnarray}
E(z)&=&u\sum_{n\geq 0}\epsilon_nv^n,\nonumber\\
g(z)&=&\sum_{n\geq 0}\gamma_nv^n,\label{formalseriesEgh}\\
h(z)&=&u\sum_{n\geq 0}\rho_nv^n,\nonumber,
\end{eqnarray}
and $\epsilon_0=\gamma_0=1$, $\rho_0=-1$. See \cite{EGinv} for proofs and more properties.

\begin{definition}
1. A function $f : \Omega \rightarrow \TT$ shall be called \emph{rigid analytic on} $\Omega$ if, for each $n \geq 0$, the restriction of $f$ to $\Omega_n$ is the uniform limit of a sequence of rational functions in $\TT(z)$ with no poles in $\Omega_n$. The set of such functions is denoted $\Hol(\Omega,\TT)$.

2. A rigid analytic function $f:\Omega \rightarrow \TT$ shall be called \emph{tempered (at infinity)} if there exists a non-negative integer $n$ such that $u(z)^n f(z) \rightarrow 0$ as $|z|_\Im \rightarrow \infty$. We write $\Thol(\Omega,\TT)$ for the space of all such functions.

3. For all positive integers $l$, we extend the definitions of rigid analytic and tempered to vector valued functions $\mathcal{F} : \Omega \rightarrow \TT^l$ by requiring that each coordinate function be rigid analytic or tempered, respectively, and we write $\Hol(\Omega, \TT^l)$ and $\Thol(\Omega,\TT^l)$, respectively.
\end{definition}

\subsubsection{Convention}
For a tempered, periodic, rigid analytic function $\Hcal \in \Hol(\Omega, \TT^l)$, we shall write $\Hcal \in \TT((u))^l$ (or $\TT[[u]]^l, u\TT[[u]]^l$, etc\ldots) to mean that there exists a formal series $\mathcal{G} \in \TT((u))^l$ (or $\TT[[u]]^l, u\TT[[u]]^l$, etc\ldots) and $n\geq 0$ with $u^n\Hcal(z) = u^n\Gcal(z)$ for all $|z|_\Im$ big enough (that is to say, for all $z$ such that $|u(z)|$ is small enough). Explicitly, if we write
$$\mathcal{G}=\left(\begin{matrix} g_1 \\ \vdots \\ g_l\end{matrix}\right)$$
with $$g_i=\sum_jg_{i,j}u^i,\quad g_{i,j}\in\TT,$$
then $u(z)^nh_i(z)=u(z)^n\sum_{j}g_{i,j}u(z)^j$ for all $i$ and for all $z\in\Omega$ with $|z|_{\Im}$ big enough.
The representative $\Gcal$ determines $\Hcal$ uniquely since $\Omega$ is a connected rigid analytic space.



\subsection{Matrix uniformizers and $\chi_t$-quasiperiodicity}
This section summarizes results following from the work done in \cite{FPRP}.
For all $j \geq 0$, we define $D_0=1$ and  $D_j := (\theta^{q^j} - \theta^{q^{j-1}})(\theta^{q^j} - \theta^{q^{j-2}})\cdots (\theta^{q^j} - \theta)$, i.e. the product of all elements in $A_+$ of degree equal to $j$, see \cite[3.1.6]{Gbook}.
We set, for all $z\in\CC_\infty$, $$\mathfrak{e_c}(z) := \sum_{j \geq 0} D_j^{-1} (\pitilde z)^{q^j}.$$ Note that,
for all $z\in\CC_\infty\setminus A$, $u(z)$ is well-defined, non-zero, and 
\[\mathfrak{e_c}(z) = u(z)^{-1}.\] 
It is easy to prove that $\mathfrak{e_c}$ is equal to $\pitilde$ multiplied by the exponential uniquely associated to the lattice $A \subset \CC_\infty$, that is, for all $z\in\CC_\infty$:
$$\mathfrak{e_c}(z)=\widetilde{\pi}z \sideset{}{'}\prod_{a\in A}\left(1-\frac{z}{a}\right).$$

Recall that the \emph{Anderson generating function for the Carlitz module} is defined for each $z \in \CC_\infty$ as
\begin{align}  \label{AGFdef}
f_t(z) := \sum_{j \geq 0} \mathfrak{e_c}(z \theta^{-j-1}) t^j = \sum_{j \geq 0} \frac{1}{D_j} \tau^j\left(\frac{\pitilde z}{\theta-t}\right) = \sum_{j\geq0}\frac{\widetilde{\pi}^{q^j}z^{q^j}}{(\theta^{q^j}-t)D_j}\end{align}
by $\FF_q[t]$-linearity, we see that
$$\tau^j\left(\frac{\widetilde{\pi}z}{\theta-t}\right)=\frac{\widetilde{\pi}^{q^j}z^{q^j}}{\theta^{q^j}-t}.$$
The second equality is due to the first author and demonstrates that $f_t$ is a $\TT$-valued rigid analytic map on $\CC_\infty$ for any choice of $z\in\CC_\infty$. 

Let
\begin{align} \label{omegadef}
\omega(t) := f_t(1) = \sum_{j\geq0}\frac{\widetilde{\pi}^{q^j}}{(\theta^{q^j}-t) D_j}\end{align}

As used in \cite{FPRP}, the function $z \mapsto \omega(t)^{-1} f_t(z)$ on $\CC_\infty$ gives an $\FF_q$-linear $\TT$-valued rigid analytic extension of the character $\chi_t : A \rightarrow \FF_q[t]$ defined above. Thus it makes sense to write
\[\chi_t(z) := \omega(t)^{-1} f_t(z), \text{ for all } z \in \CC_\infty.\]

We recall that a $\TT$-valued function $\phi:\CC_\infty\rightarrow\TT$ is called \emph{$\EE$-entire}, if for all $z \in \CC_\infty$, the image function $\phi(z) \in \CC_\infty[[t]]$ is entire in the variable $t$ in the sense of 
ibid.; the following more precise result about $\chi_t$ is proved therein.

\begin{lemma} \label{chipropslem}
The map $\chi_t : \CC_\infty \rightarrow \TT$ is in fact $\EE$-entire.
It is the unique $\FF_q$-linear, $\EE$-entire function satisfying $\chi_t(a) = a(t)$ for all $a \in A$, and
\begin{align} \label{chigrowtheq}
\| \chi_t(z) \| \leq \max\{ 1, |\mathfrak{e_c}(z)|^{\frac{1}{q}} \} \ \text{ for all } z \in \CC_\infty.\end{align}
Further, 
\begin{align}
\label{chitevals} \chi_t(z)|_{t = \theta^{q^j}} = z^{q^j}, \quad \forall j \geq 0, \forall z \in \CC_\infty , \text{ and} 
\end{align}
$\chi_t$ satisfies the following $\tau$-difference equation,
\begin{align} \label{taudifferencechi}
\tau(\chi_t)(z) = \chi_t(z) + \frac{1}{u(z)\tau(\omega(t))}.
\end{align}
\end{lemma}

\subsubsection{$\chi_t$-quasiperiodicity} \label{qperiodsect}
Recall that there is a group homomorphism 
\[ \alpha : A \rightarrow \GLA \text{ given by } a \mapsto \left( \begin{smallmatrix} 1 & a \\ 0 & 1 \end{smallmatrix} \right),\] 
and denote the image of $A$ under this map by $\Gamma_A$.

In connection, one may study in \cite{FPRP} the $\EE$-entire matrix function 
\[\Theta_t(z) := \left( \begin{matrix} 1 & 0 \\ -\chi_t(z) & 1 \end{matrix} \right),\]
which satisfies 
\[\Theta_t(a) = \rho^*_t(\alpha(a)),\] for all $a \in A$.

We state the following result, which is a direct consequence of the work done in \cite{FPRP}, as motivation for Definition \ref{VMFdef} and Remark \ref{matrixuniformizer} below. 
It is of crucial importance in all that we do to follow, and it demonstrates the role that $\chi_t$ plays in giving rise to a uniformizer for the cusp at infinity for the vectorial modular forms defined below.

\begin{lemma} \label{periodicitylem}
Suppose that $\Hcal : \Omega \rightarrow \TT^2$ is a tempered rigid analytic function such that
\begin{equation}\label{quasieq}
\Hcal(z+a) = \Theta_t(a)\Hcal(z) \text{ for all } a \in A \text{ and } z \in \Omega.
\end{equation}
Then $\Theta_t^{-1}\Hcal \in \TT((u))^2$. \hfill $\qed$
\end{lemma}

\begin{remark}
We call the coordinate functions of $\Hcal$ as in Lemma \ref{periodicitylem} above \emph{$\chi_t$-quasiperiodic}. Observe that the first coordinate is additionally $A$-periodic.
\end{remark}

\subsubsection{Uniformizers}
Recall the function 
\[\psi_1(z) := \frac{1}{\pitilde} \sum_{a \in A} \frac{\chi_t(a)}{z-a},\] 
introduced in \cite{RPmathz} and defined for all $z \in \CC_\infty \setminus A$.  The following identity holds in $\Hol(\Omega, \TT)$:
\begin{equation} \label{perkinsid}
\psi_1 = u \chi_t. 
\end{equation}
as shown in \cite[Theorem 1.1]{RPmathz}; see \cite{FPRP} for an alternate proof and a generalization to an arbitrary number of variables $t_1,\dots,t_s$.

Further we set $$\Upsilon := \left(\begin{matrix} 1 & 0 \\ 0 & u \end{matrix} \right),\quad 
\Psi_1 := \left(\begin{matrix} 1 & 0 \\ -\psi_1 & 1 \end{matrix} \right).$$
We have the following immediate, but useful, consequence of Lemma \ref{periodicitylem}.

\begin{corollary} \label{infequiv1}
Let $\Hcal$ be an element of $\Thol(\Omega,\TT^2)$ satisfying \eqref{quasieq}. The following conditions are equivalent:
\begin{align} \label{vanishcond} & (\Upsilon \Hcal)(z) \rightarrow 0 \text{ as } |z|_\Im \rightarrow \infty,  \\
\label{uexpcond} & \Upsilon \Theta^{-1}_t \Hcal \in u\TT[[u]]^2.
\end{align}
\end{corollary}

\begin{proof}
The equivalence follows from Lemma \ref{periodicitylem} and the identity
\begin{equation}\label{upsiloncommeq}
\Upsilon \Theta^{-1}_t = \left(\begin{matrix} 1 & 0 \\ 0 & u \end{matrix} \right)\left(\begin{matrix} 1 & 0 \\ \chi_t & 1 \end{matrix} \right) = \left(\begin{matrix} 1 & 0 \\ \psi_1 & 1 \end{matrix} \right)\left(\begin{matrix} 1 & 0 \\ 0 & u \end{matrix} \right) = \Psi_1^{-1} \Upsilon,\end{equation}
a consequence of \eqref{perkinsid}.


\end{proof}




\section{$\TT$-valued vectorial modular forms}\label{valuedvectorialmodularforms}

\subsection{$\TT$-valued Drinfeld modular forms}
Before defining our $\TT$-valued vectorial modular forms, we recall the one-dimensional results from \cite{FPannals} which we use in the sequel.

Recall the factor of automorphy,
\[j(\gamma,z) := cz + d,\] 
defined and nonzero for all $\gamma = \left( \begin{smallmatrix} a & b \\ c & d  \end{smallmatrix} \right) \in \GLA$ and $z \in \Omega$.

\begin{definition}
A \emph{weak $\TT$-valued modular form} of weight $k$ and type $m \pmod{q-1}$ for $\GLA$ is a function $f \in \Thol(\Omega, \TT)$ such that
\[f(\gamma (z)) = \frac{j(\gamma,z)^k}{\det\gamma^m} f(z) \text{ for all } \gamma \in \GLA, z \in \Omega.\]
In particular, it follows that $f \in \TT((u))$. The set of these functions is a $\TT\otimes_{\CC_\infty}\CC[j]$-module, where $j=\frac{g^q}{\Delta}$ is the Drinfeld modular $j$-invariant in the notations of Gekeler, \cite{EGinv}.
We denote this module by $$\MM_k^m(\boldsymbol{1})^!.$$ The $\CC_\infty[j]$-submodule of $\CC_\infty$-valued forms are denoted $M_k^m(\boldsymbol{1})^!$. 

Further, we denote by $$\MM_k^m(\boldsymbol{1})$$ the $\TT$-submodule of 
$\MM_k^m(\boldsymbol{1})^!$ whose elements belong to $\TT[[u]]$, and $M_k^m(\boldsymbol{1})$ denotes 
the $\CC_\infty$-sub-vector space spanned by the forms which are $\CC_\infty$-valued. 

An element of $\MM_k^m(\boldsymbol{1})$ is called {\em $\TT$-valued modular form} (of weight $k$ and type $m\pmod{q-1}$). Finally, any element of $\MM_k^m(\boldsymbol{1})\cap u\TT[[u]]$ is called \emph{cuspidal} of weight $k$ and type $m \pmod{q-1}$ (or cusp form). The cusp forms of weight $k$, type $m \pmod{q-1}$ span a $\TT$-submodule
of $\MM_k^m(\boldsymbol{1})$ which is denoted by $\mathbb{S}_k^m(\boldsymbol{1})$, and 
the $\CC_\infty$-sub-vector space of the $\CC_\infty$-valued cusp forms is denoted by $S_k^m(\boldsymbol{1})$.
\end{definition}

\begin{remark}
We note that, more generally, one may consider Drinfeld modular forms with values in a general $\CC_\infty$-Banach algebra such as the algebras $\mathbb{B}_s$ considered by Angl\`es, Tavares-Ribeiro and the first author in \cite{APTRandstark} or the multivariate Tate algebras $\TT_s$ which are now ubiquitous to the theory. 
\end{remark}

We have the following result, due to the first author in \cite[Lem. 13]{FPannals}.

\begin{lemma}\label{isomorphismtrivialchar}
For all non-negative integers $k$ and classes $m \pmod{q-1}$, we have
\[\MM_{k}^{m}(\boldsymbol{1})^! \cong M_k^m(\boldsymbol{1})^! \otimes_{\CC_\infty} \TT. \]
\end{lemma}
We deduce that
$$\MM_{k}^{m}(\boldsymbol{1}) \cong M_k^m(\boldsymbol{1}) \otimes_{\CC_\infty} \TT,\quad 
\mathbb{S}_{k}^{m}(\boldsymbol{1}) \cong S_k^m(\boldsymbol{1}) \otimes_{\CC_\infty} \TT.$$

It is well known that the set $M=M(\boldsymbol{1})$ of all the $\CC_\infty$-valued Drinfeld modular forms 
for the group $\GLA$ is a $\CC_\infty$-algebra which is graded by weights and types (the group
$\mathbb{Z}\times\mathbb{Z}/(q-1)\mathbb{Z}$), and that
$$M=\bigoplus_{k\in\mathbb{Z},\atop m\in\frac{\mathbb{Z}}{(q-1)\mathbb{Z}}}M_k^m(\boldsymbol{1})=\CC_\infty[g,h]$$
(see Gekeler, \cite{EGinv}). We further denote by $\MM(\boldsymbol{1})$, or more simply, by $\MM$, the $\TT$-algebra $\TT\otimes_{\CC_\infty}M$. Then, we deduce from Lemma \ref{isomorphismtrivialchar} that 
$$\MM=\TT[g,h].$$

Now we may define the main object of our interest.

\begin{definition} \label{VMFdef}
1. We say that a function $\Hcal \in \Thol(\Omega, \TT^2)$ is a $\TT$-valued \emph{weak vectorial modular forms} (abbreviated VMF$^!$) of \emph{weight} $k$ and \emph{type} $m \pmod{q-1}$ for $\rho^*_t$ if the following condition is satisfied:
\begin{equation} \label{modularity}
 \Hcal (\gamma (z)) = \frac{j(\gamma,z)^k}{\det \gamma^m} \rho^*_t(\gamma) \Hcal(z), \ \text{ for all } \gamma \in \GLA \text{ and }  z \in \Omega. \\
\end{equation}
We denote the $\TT\otimes_{\CC_\infty}\CC_\infty[j]$-module of such forms by $\MM^{m}_{k}(\rho^*_t)^!$. \medskip

2. Recall that $\Upsilon := \left(\begin{smallmatrix} 1 & 0 \\ 0 & u \end{smallmatrix} \right)$. If
$$ (\Upsilon\Hcal)(z) \rightarrow \binom{0}{0}$$ as $|z|_\Im \rightarrow \infty$ (see the equivalent conditions of Cor. \ref{infequiv1}),
we shall say that $\Hcal$ is a \emph{vectorial modular form} (abbreviated VMF) of weight $k$, type $m$ and representation $\rho_t^*$. We denote the $\TT$-module of such forms by $\MM_{k}^{m}(\rho^*_t)$. \medskip

3. If, additionally,
\begin{equation} \label{cuspidal}
\Hcal(z) \rightarrow \binom{0}{0} \text{ as } |z|_\Im \rightarrow \infty,
\end{equation}
we shall call $\Hcal$ a \emph{cuspidal} VMF (of the same weight, type and representation as above). We denote the $\TT$-module of such forms by $\SSS_{k}^m(\rho^*_t)$.
\end{definition}

It is easy to see that the $\TT$-span of all the $\TT$-valued vectorial modular forms
of $\MM_k^m(\rho_t^*)$ for all possible choices of $k$ and $m$ is a $\MM$-module that we denote by 
$$\MM(\rho_t^*).$$ Of course,
this module is graded by the weights and the types:
$$\MM(\rho_t^*)=\bigoplus_{k,m}\MM_k^m(\rho_t^*),$$ it is in fact a graded module over
the graded algebra $\MM$.

\begin{remark} \label{matrixuniformizer}
The condition $\Upsilon \Hcal \rightarrow \binom{0}{0}$ as $|z|_\Im \rightarrow \infty$ (condition \eqref{vanishcond} above) may seem arbitrary at first glance, however, it was chosen very carefully as the ``weakest'' condition for which our $\TT$-modules of vectorial forms of a given weight and type have finite rank and are stable under Anderson twists, stable under Hecke operators and include the Eisenstein series. 

We shall discuss below several conditions equivalent with those coming from Cor. \ref{infequiv1} which we hope makes it clear that ours is a good notion of regularity at the cusp at infinity for the VMF just defined.
We also notice that our condition \eqref{vanishcond} is intimately related with the choice of the 
representation $\rho_t^*$. There are several other representations $\rho:\GLA\rightarrow\GL_2(\TT)$
and to each one we can of course associate $\TT\otimes_{\CC_\infty}\CC_\infty[j]$-modules $\MM_k^m(\rho)^!$; however, we ignore, in such a level of generality, what could be the good analogue of condition \eqref{vanishcond}.
\end{remark}

We shall often need to examine the individual coordinate functions of a VMF, and we introduce the following notation.

\subsubsection*{Notation}
For a column vector $\Hcal = (h_1, h_2, \cdots, h_l)^{tr}$, we write $[\Hcal]_i = h_i$, for $i= 1,2,\dots,l$. 

\subsection{Examples and non-examples}
\subsubsection{ Vectorial Eisenstein series} \label{Eisseriessec}
We define the \emph{vectorial Eisenstein series} of weight $k$ for $\rho^*_t$ by
\[\Ecal_k(z) := \sideset{}{'}\sum_{a,b \in A} (az+b)^{-k} \left( \begin{matrix} a(t) \\ b(t) \end{matrix}  \right),\]
with the primed summation indicating the absence of the term where both $a,b$ are zero. These functions may also be defined intrinsically as in \cite[\S 2.4]{FPannals}.

It follows by the work done in \cite[Prop. 22]{FPannals} that $\Ecal_k \neq 0$ for all $k \equiv 1 \pmod{q-1}$, and $\Upsilon \Ecal_k \rightarrow \binom{0}{0}$ but $\Ecal_k \not\rightarrow \binom{0}{0}$ as $|z|_\Im \rightarrow \infty$. Thus, for such $k$ we have $\mathcal{E}_k \in \MM_k^0(\rho^*_t) \setminus \SSS_k^0(\rho_t^*)$.

Below, we shall demonstrate that these Eisenstein series are Hecke eigenforms with explicit eigenvalues, and we shall have much to glean from their representation at the cusp at infinity $\Ecal_k \in \Theta_t \TT[[u]]$. Further, we observe that the map $\tau:\TT\rightarrow\TT$ induces 
a (homogeneous) endomorphism 
$$\tau:\MM\rightarrow\MM.$$
Hence, we can define the skew ring $\MM\{\tau\}$ whose elements are the finite sums
$$\sum_{i\geq 0}f_i\tau^i,\quad f_i\in\MM,$$ with the product defined by the rule $\tau f=\tau(f)\tau$
for $f\in \MM$. Then, it is easy to see that $\MM(\rho_t^*)$ is equipped with the structure of 
a $\MM\{\tau\}$-module, and we prove below that $\Ecal_1$ cyclically generates this module.

\subsubsection{Anderson generating functions in rank 2}
The following example, which lies just outside the confines of Definition \ref{VMFdef}, arises from Anderson's theory of $t$-motives and rigid analytic trivializations; see \cite{FPbourbaki}.

We recall from \eqref{omegadef} the {\em Anderson-Thakur function}, given here in its equivalent product form,
$$\omega(t)=\iota_\theta\prod_{i\geq 0}\left(1-\frac{t}{\theta^{q^i}}\right)^{-1}\in\TT^\times,$$ 
where $\iota_\theta$ that is the same element of Carlitz $\theta$ torsion as in \eqref{pitildedef}; see \cite{APinv} for
a recent overview on the properties of this function. 
We note that $1/\omega$ is an entire function of $t$ on $\CC_\infty$ and vanishes, as a function of the variable $t$, if and only if $t\in \{\theta,\theta^q,\theta^{q^2},\ldots\}$.

Let $z \in \Omega$. By the general theory for Drinfeld modules (see e.g. \cite[Ch. 4]{Gbook}), associated to the lattice $zA + A \subset \CC_\infty$, one has the entire exponential function
$\mathfrak{e}^z:\CC_\infty\rightarrow\CC_\infty$ defined by
 $$\mathfrak{e}^z(w) := \sum_{i\geq 0}\alpha_i(z) \tau^i(w)$$ with $\alpha_i\in M_{q^i-1}^0$; 
the functions $\alpha_i$ have been referred to by Gekeler as \emph{para-Eisenstein series} \cite{EGpes11} due to their similarities with the Eisenstein series $E^{(q^k - 1)}$ of Goss. NB. that our normalizations are slightly different than Gekeler's, as we consider the exponential for the lattice $Az+A$, while he considers the exponential for $\pitilde(Az+A)$. 
 
 In \cite{FPcrelles,FPannals}, the normalized Anderson generating functions
\begin{align} \label{d1d2def}
\bm{d}_1(z;t) := \frac{1}{\omega(t)}\sum_{i \geq 0} \frac{\alpha_i(z)}{\theta^{q^i} - t} z^{q^i} \text{ and } \bm{d}_2(z;t) := \frac{1}{\omega(t)}\sum_{i \geq 0} \frac{\alpha_i(z)}{\theta^{q^i} - t} \end{align}
were introduced and studied. These series both define holomorphic functions $\Omega\rightarrow\EE$. 
We recall from \cite[Eq. (2.25)]{FPcrelles} that
\begin{equation}\label{eqd2}
\bm{d}_2=1-(t-\theta)[u^{q-1}-u^{(q-1)(q^2-q+1)}+u^{(q-1)q^2}+o(u^{(q-1)q^2})],
\end{equation}
in $A[t][[u^{q-1}]]$; this series converges for the Gauss norm of $\TT$ whenever $|u|$ is small enough. In ibid.
the first six non-zero terms of this series expansion have been computed\footnote{This long and difficult computation was pursued by V. Bosser.}.

\subsubsection{The function $d_3$}
We observe that the function 
\[\dd_3=\bm{d}_1-\chi_t(z)\bm{d}_2,\]
which is an entire function $\CC_\infty\rightarrow\EE$, is $A$-periodic; i.e.,
$\bm{d}_3(z+a)=\bm{d}_3(z)$ for all $a\in A$. We shall give a short account 
of the main properties of this function, without giving full details.

Setting
\begin{eqnarray*}
\boldsymbol{\psi}&:=&\tau(\omega)^{-1}\left(\frac{1}{u}\left(\bm{d}_2+\frac{\bm{d}_2-g\bm{d}_2^{(1)}}{(t-\theta^q)u^{q-1}}\right)\right)\\
&=&\tau(\omega)^{-1}u^{q-2}\{(\theta-t)+u^{(q-1)^2}+(\theta-\theta^q)u^{(q-1)q}+o(u^{(q-1)q})\},\end{eqnarray*}
it is quite simple, although rather intricate in the computations, to show that
$\bm{d}_3$ further satisfies a linear, non-homogeneous $\tau$-difference equation of order $2$:
 \begin{equation}\label{eqdelta2}
\bm{d}_3=(t-\theta^q)\Delta\tau^2(\bm{d}_3)+g\tau(\bm{d}_3)+\boldsymbol{\psi}.
\end{equation}
From this, we deduce that the $u$-expansion of $\tau(\omega)\bm{d}_3$ begins, for $q>2$, with the following terms:
\begin{equation}\label{qdifferent2}
-u^{q-2}(t-\theta+(t-\theta)u^{q(q-1)^2}+o(u^{q(q-1)^2})).
\end{equation}
If $q=2$, the $u$-expansion of $\tau(\omega)\bm{d}_3$ begins with the following terms:
\begin{equation}\label{qequal2}t+\theta+(1+t+\theta)u^2+o(u^2).
\end{equation}
Further, if $q>2$, we have the limit $\lim_{t\to\theta}\bm{d}_3=0$ for all $z\in\Omega$ and $\bm{d}_3$ is the only solution of (\ref{eqdelta2}) with this property. Note that if $q=2$, $\bm{d}_3$ does not vanish at $u=0$.
In all cases, it can be proved that $\bm{d}_3$ is not a modular form.

\subsubsection{The Legendre period form} \label{nonexample}
It follows from the work of the first author that the vectorial function $\Fcal^*:\Omega\rightarrow\EE^2$
defined by
$$\Fcal^* := {-\bm{d}_2 \choose \bm{d_1}}$$ is in $\MM_{-1}^{-1}(\rho_t^*)^!$.

We observe, by the definition of $\bm{d}_3$, that $$\Fcal^* = \Theta_t {-\bm{d}_2 \choose \bm{d_3}} \in \Theta_t \TT[[u]]^2.$$
However, by (\ref{eqd2}), $\Fcal^*$ fails condition \eqref{vanishcond}. In other words, $$\Fcal^*\in\MM_{-1}^{-1}(\rho_t^*)^!\setminus\MM_{-1}^{-1}(\rho_t^*).$$

More generally, if $k\geq q-1$ and $q-1$ divides $k$, the convergent series 
\begin{align} \label{GossESdef}
E^{(k)}(z) := \sideset{}{'}\sum_{a,b\in A}(az+b)^{-k}, \end{align}
{\em Goss' Eisenstein series} of weight $k$, with the notation of \cite[(5.9)]{EGinv}, is in $M_k^0\setminus\{0\}$
and the function $E^{(k)}\Fcal^*$ is an element of 
$\MM_{q-2}^{-1}(\rho_t^*)^!\setminus\MM_{q-2}^{-1}(\rho_t^*)$.
\begin{remark} As we shall demonstrate below, many problems arise for elements of $\MM_{k}^{-m}(\rho_t^*)^!$ whose first coordinate function does not vanish at infinity (which is only possible for type $m \equiv -1 \pmod{q-1}$). The requirement that such a form is representable in $\Theta_t \TT[[u]]^2$ is not a property which is stable under the action of the operator $\tau$ or stable under the action of the Hecke operators that we are going to introduce later, when the first coordinate function is non-vanishing, whereas the condition that the form is representable in $\Theta_t\Upsilon^{-1}(u\TT[[u]]^2)$ is stable for these actions. We shall have more to say about this below in the appropriate sections.\end{remark}





\subsection{Module structures}
In this section we determine the explicit structure of $\MM_k^m(\rho_t^*)$ for $k \geq 0$ and $m\in\frac{\mathbb{Z}}{(q-1)\mathbb{Z}}$ as a $\TT$-module. To begin with, we give here an $A$-expansion for $\Ecal_1$ which was the starting point of this collaboration. We generalize the next result to all non-zero $\Ecal_k$ in Theorem \ref{eisAexp}. 

\begin{proposition} \label{AexpE1prop}
The following expansions hold for $|z|_\Im$ sufficiently large:
\begin{flalign*}  \Ecal_1(z) &= {0 \choose -L(\chi_t,1)} + \pitilde \sum_{a \in A_+} { - a(t) \choose \chi_t(az)} u(az), \\
\Ecal_q(z) &=  {0 \choose -L(\chi_t,q) + \frac{\pitilde^q}{\tau(\omega)}\sum_{a \in A_+} u(az)^{q-1} } + \pitilde^q \sum_{a \in A_+} { - a(t) \choose \chi_t(az)} u(az)^q .
\end{flalign*}
\end{proposition}

\begin{proof} The convergence of all the series involved in the above identities for $|z|_{\Im}$ big enough
is easily checked by using Lem. \ref{chipropslem}.
The first coordinate in both $\Ecal_1$ and $\Ecal_q$ was already handled in \cite[Lem. 21]{FPannals}. We focus on the second coordinate in $\Ecal_1$ and employ the identity 
\[\sum_{b \in A} \frac{b(t)}{z+b} = -\pitilde u(z) \chi_t(z),\]
which follows from \cite[Th. 1.1]{RPmathz}.
We have
\begin{flalign*}
\sideset{}{'}\sum_{a,b \in A} \frac{b(t)}{az+b} &= -L(\chi_t,1) -\sum_{a \in A_+} \sum_{b \in A} \frac{b(t)}{az+b} \\
&= -L(\chi_1,1) + \pitilde \sum_{a \in A_+} \chi_t(az) u(az).
\end{flalign*}
This gives the identity above for $\Ecal_1$.

Applying $\tau$ to both sides of the identity for $\Ecal_1$ and using Lem. \ref{chipropslem} finishes the proof.
\end{proof}

\begin{remark} We denote by $A^+$ the subset of $A$ of monic polynomials.
It is easy to see, from \cite[Examples (6.4)]{EGinv}, that the element
$$F=\frac{1}{\theta^q-t}-\sum_{a\in A^+}u(az)^{q-1}\in \TT[[u]]$$ can be evaluated at $t=\theta$
and satisfies
$$F|_{t=\theta}=\widetilde{\pi}^{1-q}E^{(q-1)}.$$ 
We recall that
$$L(\chi_t,1)=-\frac{\widetilde{\pi}}{\tau(\omega)},$$ so that, applying $\tau$:
$$L(\chi_t,q)=\frac{\widetilde{\pi}^q}{(\theta^q-t)\tau(\omega)}.$$
Since we have, for the second
component of the ``constant" term of the previous expansion of $\mathcal{E}_q$:
$$-L(\chi_t,q) + \frac{\pitilde^q}{\tau(\omega)}\sum_{a \in A_+} u(az)^{q-1}=-\frac{\widetilde{\pi}^q}{\tau(\omega)}\left(\frac{1}{\theta^q-t}-\sum_{a\in A^+}u(az)^{q-1}\right),$$ we notice that this coefficient, although not 
a modular form, evaluates at $t=\theta$ to the Eisenstein series $E^{(q-1)}$.

\end{remark}

The proof of the next result mirrors \cite[Prop. 19]{FPannals}. We recall that $\tau(\mathcal{E}_1)=\mathcal{E}_q$.

\begin{theorem} \label{structurethm}
For each positive integer $k$ and each $m \pmod{q-1}$, we have
\[\MM_k^m(\rho^*_t) \cong \MM_{k-1}^m \mathcal{E}_1 \oplus \MM_{k-q}^m \tau(\mathcal{E}_1),\]
as $\TT$-modules. Further, $\MM_k^m(\rho^*_t)\neq\{0\}$ if and only if $k\equiv 2m-1\pmod{q-1}$.
\begin{proof}
Consider the matrix $$\bm{E} := (\mathcal{E}_1, \tau(\mathcal{E}_1))\in\operatorname{Mat}_{2\times 2}(\Hol(\Omega, \TT)).$$ One easily sees that $\det(\bm{E})$ lies in $\MM_{q+1}^1(\boldsymbol{1})^!$ (see e.g.  \cite[Lem. 14]{FPannals}), and from Proposition \ref{AexpE1prop} and Lem. \ref{chipropslem} we see that this determinant is holomorphic at infinity. Hence, $\det(\bm{E})$ is a $\TT$-multiple of Gekeler's $$h = -u + o(u)^2.$$ Because $ \MM_{q+1}^1(\GL_2(A))$ is free of rank 1 as a $\TT$-module, to determine $\det(\bm{E})$ it suffices to determine the coefficient of $u$ in $\det(\bm{E})$. Using Proposition \ref{AexpE1prop} again, we see that this equals $\pitilde L(\chi_t,q)$. Hence, $$\det(\bm{E}) = -\pitilde L(\chi_t,q) h.$$ Finally, since $\Delta=-h^{q-1}$
(by \cite[Theorem (9.1)]{EGinv}), we see that $\det(\bm{E})$ does not vanish for any $z \in \Omega$. 

We prove that the left matrix multiplication by $\bm{E}$ gives an isomorphism $$\MM_{k-1}^m \oplus \MM_{k-q}^m \rightarrow \MM_k^m(\rho^*_t),$$ for each positive integer $k$ and class $m \pmod{q-1}$. Let us consider an element $$\mathcal{H}  = \left(\begin{smallmatrix} h_1 \\ h_2 \end{smallmatrix} \right) \in \MM_{k-1}^m \oplus \MM_{k-q}^m.$$ Clearly, $$\bm{E}\cdot\mathcal{H} = \Ecal_1 h_1 + \Ecal_q h_2 \in \MM_k^m(\rho^*_t)^!,$$ and it suffices to check $\eqref{vanishcond}$. We have
\[\Upsilon\bm{E}\cdot\mathcal{H} = h_1 \Upsilon \mathcal{E}_1 + h_2 \Upsilon \tau(\mathcal{E}_1),\]
and hence $\Upsilon\bm{E}\cdot\mathcal{H} \rightarrow 0$, as $|z|_\Im \rightarrow \infty$
because we have, from Proposition \ref{AexpE1prop}, that $\Upsilon \mathcal{E}_1,\Upsilon \mathcal{E}_q\rightarrow\binom{0}{0}$, and $h_1(z),h_2(z)$ are obviously bounded for $|z|_{\Im}$ big enough.

Conversely, let us choose $$\mathcal{G} = \left(\begin{smallmatrix} g_1 \\ g_2 \end{smallmatrix} \right) \in \MM_k^m(\rho^*_t).$$ From \eqref{upsiloncommeq} and \eqref{uexpcond}, we obtain $\Upsilon \mathcal{G} \in u\Psi_1 \TT[[u]]^2$. 
Writing out explicit entries for $\bm{E}$, one sees that the entries of $(\Upsilon \bm{E})^{-1}$ grow at most like $u^{-1}$ at infinity. Thus we see that the entries of $\bm{E}^{-1} \mathcal{G} = (\Upsilon \bm{E})^{-1} \Upsilon \mathcal{G}$ are bounded at infinity. The non-vanishing condition of the Theorem now follows
from \cite[Remark (5.8)]{EGinv}, finishing the proof.
\end{proof}
\end{theorem}

\begin{remark}
Clearly now, $\MM_k^m(\rho_t^*)$ is a free $\TT$-module of finite rank for all $k \geq 0$ and $m \pmod{q-1}$.
\end{remark}

The following corollary follows immediately.

\begin{corollary}\label{taustability}
Suppose $\mathcal{H} \in \MM^m_k(\rho^*_t)$, then $\tau(\mathcal{H}) \in \MM^m_{kq}(\rho^*_t)$.
\end{corollary}

\begin{remark}
Here we prove Corollary \ref{taustability} directly.

Suppose $\Hcal = \Theta_t {h_1 \choose h_2} \in \Theta_t \TT[[u]]^2$. From Lem. \ref{chipropslem}, we obtain
\[\tau(\Hcal) = \left(\Theta_t + \left(\begin{smallmatrix} 0 & 0 \\ (u\tau(\omega))^{-1} & 0 \end{smallmatrix} \right)\right) \left(\begin{smallmatrix} \tau(h_1) \\ \tau(h_2) \end{smallmatrix} \right) = \Theta_t\left(\begin{smallmatrix} \tau(h_1) \\ (u\tau(\omega))^{-1}\tau(h_1) + \tau(h_2) \end{smallmatrix} \right).\]
We deduce that, if $h_1\in\TT[[u]]$, then 
$\tau(\Hcal)\in\Theta_t \TT[[u]]^2$ if and only if $h_1 \in u\TT[[u]]$. 
It follows that, for all non-zero Eisenstein series $E^{(k)}$ as in \eqref{GossESdef} above, 
\[[E^{(k)}\Fcal^*]_1\in\TT[[u]]\setminus u\TT[[u]].\]
\end{remark}

\begin{corollary}
The $\MM\{\tau\}$-module $\MM(\rho^*_t)$ is cyclic, generated by $\mathcal{E}_1$.
\end{corollary}

As another corollary, we re-obtain the main result of \cite{FPannals}, namely Theorem 8 there.

\begin{corollary} \label{e1totaufstar}
We have $\Ecal_1 = -\pitilde h \tau(\Fcal^*)$. In particular, this gives analytic continuation to $\Ecal_1$ in the variable $t$ to all of $\CC_\infty$. 
\end{corollary}

\begin{proof}
As we have noticed above, we may write $$\Fcal^* = {-\bm{d}_2 \choose \bm{d}_3 + \chi_t\bm{d}_2},$$ for some $\bm{d}_3 \in \TT[[u]]$. Thus, $$\tau(\Fcal^*) = {-\tau(\bm{d}_2) \choose \tau(\bm{d}_3) + \chi_t\tau(\bm{d}_2) + (\tau(\omega)u)^{-1}\tau(\bm{d}_2)},$$ as follows from Lem. \ref{chipropslem}. Now, $$-h\tau(\bm{d}_2) = u + o(u^2) \in \TT[[u]]$$ and both $h\tau(\bm{d}_3)$ and $h\chi_t\tau(\bm{d}_2)$ vanish as $|z|_\Im \rightarrow \infty$. Hence, $h\tau(\Fcal^*) \in \MM_1^0(\rho_t^*)$. By the previous theorem, we must have $h\tau(\Fcal^*)$ is a $\TT$-multiple of $\Ecal_1$. Comparing the coefficients of $u$ in the first coordinates, we deduce the result.
\end{proof}

\subsection{The $\tau$-difference equation for $\Ecal_1$}\label{differenceE1}
Of course, one may determine the $\tau$-difference equation satisfied by $\Ecal_1$ from Corollary \ref{e1totaufstar} and the $\tau$-difference equation for $\Fcal^*$, which is well-known (e.g. \cite[Proposition 16]{FPannals}). Here we deduce it from the finite $\TT$-rank of the modules $\MM_k^m(\rho_t^*)$, which follows from Theorem \ref{structurethm} above. We recall from \cite{EGinv} the $A$-expansions: 
\[\Delta(z) = \sum_{a \in A_+} a^{q^2 - q} u(az)^{q-1} \ \text{ and } \ g(z) = 1 - (\theta^q - \theta) \sum_{a \in A_+} u(az)^{q-1}.\]

\begin{proposition} \label{Eisdiffeq}
We have 
\[ \tau^2\left(\small{\frac{1}{L(\chi_t,1)}}\Ecal_1\right) =  \frac{(\theta^q - t) \Delta}{L(\chi_t,1)}\Ecal_1 + g^q \tau\left(\frac{1}{L(\chi_t,1)}\Ecal_1\right).\]
\end{proposition}
\begin{proof}
We have that $\MM_{q^2 - q}^0(\boldsymbol{1})$ is the $\TT$-span of $g^q$, and $\MM_{q^2 - 1}^0(\boldsymbol{1})$ is the $\TT$-span of $\Delta$ and $g^{q+1}$. Thus, it follows from Theorem \ref{structurethm} that there exist $a,b,c \in \TT$ such that $$\tau^2(\Ecal_1) = (a \Delta + b g^{q+1}) \Ecal_1 + c g^q \tau(\Ecal_1).$$ Now, by Theorem \ref{eisAexp}, we see that $[\tau^2(\Ecal_1)]_1 \rightarrow 0$ like $u^{q^2}$, as $|z|_\Im \rightarrow \infty$. Similarly, $[g^q\tau(\Ecal_1)]_1 \rightarrow 0$ like $u^{q}$, $[\Delta\Ecal_1]_1 \rightarrow 0$ like $u^{q}$, and $[g^{q+1}\Ecal_1]_1 \rightarrow 0$ like $u$. Thus, we must have $b = 0$ above. 

Now, $[\Delta \Ecal_1]_2 \rightarrow 0$, and hence we can compare the constant terms of $[\tau^2(\Ecal_1)]_2$ and $[g^q\tau(\Ecal_1)]_2$ to determine $c$. The constant term of the former is $-L(\chi_t,q^2)$ and of the latter is $-L(\chi_t,q)$. It follows from the equation $$L(\chi_t,1) = \frac{-\pitilde}{(t-\theta)\omega(t)}$$ that $$L(\chi_t,q^2) = \frac{\pitilde^{q^2 - q}}{t - \theta^{q^2}} L(\chi_t,q).$$ Hence, $$c = \frac{\pitilde^{q^2 - q}}{t - \theta^{q^2}}.$$

Finally, appealing again to Theorem \ref{eisAexp}, and observing that we must cancel the $u^q$ terms in $[g^q\tau(\Ecal_1)]_1$ and $[\Delta\Ecal_1]_1$, we conclude that $$a = -\frac{\pitilde^{q^2 - 1}}{t-\theta^{q^2}}.$$ Renormalizing gives the result.
\end{proof}

A more compact matrix equation will be useful in the sequel. Define the square matrix
\[ \bm{E}^\star := \left( \frac{1}{L(\chi_t,1)}\Ecal_1, \tau\left(\frac{1}{L(\chi_t,1)}\Ecal_1\right)\right) = \bm{E}\cdot\left( \begin{matrix} L(\chi_t,1)^{-1} & 0 \\ 0 & L(\chi_t,q)^{-1} \end{matrix} \right),\]
with $\bm{E}$ from the proof of Theorem \ref{structurethm}, and for a matrix $(f_{ij})$ with coefficients in $\Hol(\Omega, \TT)$, let 
\[\tau^k(f_{ij}) :=  ( \tau^k(f_{ij})).\]

The next result follows immediately from the last, and expresses $\tau^k(\Ecal_1/L(\chi_t,1))$ as an explicit linear combination of $\Ecal_1/L(\chi_t,1)$ and $\tau(\Ecal_1/L(\chi_t,1))$ for all $k \geq 2$. We set
\[\bm{B} := \left(\begin{matrix} 0 & (\theta^q - t)\Delta \\ 1 & g^q  \end{matrix} \right).\]

\begin{corollary}
 For all positive integers $k \geq 1$, we have
\[\tau^k(\bm{E}^\star) = \bm{E}^\star\bm{B}\tau(\bm{B})\cdots\tau^{k-1}(\bm{B}).\]
\end{corollary}


\subsection{Evaluations at $t = \theta, \theta^q, \theta^{q^2}$ and quasi-modularity} \label{qthpowersevalsect}

We recall the definition of a {\em Drinfeld quasi-modular form of weight $k$, type $m$ and depth $\leq l$} from 
the paper \cite{BvPfimrn}. This is a holomorphic function $$f : \Omega \rightarrow \CC_\infty$$ such that there also exist $A$-periodic holomorphic functions $$f_j : \Omega \rightarrow \CC_\infty,\quad  j = 0,\dots,l,$$ 
with $u$-expansions at infinity in $\CC_\infty[[u]]$, and satisfying
\[f\left( \left(\begin{matrix} a & b \\ c & d  \end{matrix}\right)z\right) = \frac{(cz+d)^k}{(ad-bc)^m} \sum_{j = 0}^l \left(\frac{c}{cz+d}\right)^j f_j(z)\]  
for all $\left(\begin{smallmatrix} a & b \\ c & d  \end{smallmatrix}\right) \in \GL_2(A)$ and $z \in \Omega$. We write $\widetilde{M}_{k,m}^{\leq l}$ for the $\CC_\infty$ vector space of functions of weight $k$, type $m$ and depth $\leq l$. 
In \cite[Theorem 1]{BvPfimrn}, Bosser and Pellarin have shown that the full $\CC_\infty$-algebra of Drinfeld quasimodular forms of all weights, types and depths is exactly the $\CC_\infty$-algebra $\CC_\infty[E,g,h]$, which is of dimension three, and is graded by the weights and filtered by the depths. 

\begin{proposition} \label{quasimodprop}
Let $\Hcal := {h_1 \choose h_2} \in \MM_k^m(\rho_t^*)$. If for some positive integer $j$ the function $\Hcal$ may be evaluated at $t = \theta^{q^j}$, then we have
\[ h_1|_{t = \theta^{q^j}} \in \widetilde{M}_{k+q^j,m+1}^{\leq q^j}.\] 
Further,
\begin{align*} h_2|_{t = \theta^{q^j}} + z^{q^j} h_1|_{t = \theta^{q^j}} \in M_{k-q^j}^{m+1-q^j}, & \hfill \text{ if $k \geq q^j$, and } \\ 
h_1|_{t = \theta^{q^j}} \in M_{k+q^j}^{m+1}, & \hfill  \text{ if $q^j > k$} . \end{align*}
\end{proposition}

\begin{proof}
One checks directly from the definitions that $$h_3 := h_2+\chi_t h_1$$ is an $A$-periodic $\TT$-holomorphic function. Further, by the assumption that $\Hcal$ is a VMF and \eqref{chigrowtheq}, we see that 
\[u(z) h_3(z) \rightarrow 0 \text{ as } |z|_\Im \rightarrow \infty.\]
Hence, $h_3 \in \TT[[u]]$. 

Assume that $\eta_1 := h_1|_{t=\theta^{q^j}}$ and $\eta_2 := h_2|_{t = \theta^{q^j}}$ are defined. Then, by \eqref{chitevals},
\[\eta_3(z) := h_3|_{t = \theta^{q^j}}(z) = \eta_2(z) + z^{q^j} \eta_1(z) \in \CC_\infty[[u]].\] 
Hence, from the definition of VMF we learn, for all $\gamma = \left(\begin{smallmatrix} a & b \\ c & d  \end{smallmatrix}\right) \in \GL_2(A)$,
\[ \eta_1(\gamma z) = \frac{j(\gamma,z)^k}{(\det\gamma)^{m+1}}  (d^{q^j} \eta_1(z) -c^{q^j}(\eta_3(z) - z^{q^j}\eta_1(z)))\]
\[= \frac{j(\gamma,z)^{k+q^{j}}}{(\det\gamma)^{m+1}}\left(\eta_1(z) + \left(\frac{c}{cz+d}\right)^{q^j}\eta_3\right).  \]
Now we notice that by \cite[Lemma 2.5]{BvPfimrn}, the function $\eta_3$ lies in $ M_{k-q^j}^{m+1-q^j}$, and this space is zero if $q^j > k$. This implies the final two assertions of the proposition.
\end{proof}

\begin{remark}
Notice that by Theorem \ref{structurethm} and the analytic continuation to $\CC_\infty$ obtained for the variable $t$ in Corollary \ref{e1totaufstar}, any VMF in the span $M_{k-1}^m \Ecal_1 + M_{k-q}^m \Ecal_q$ has entire coordinate functions in the variable $t$ and may be evaluated at $t = \theta^{q^j}$ for all $j \geq 0$. 
\end{remark}

\subsubsection{Modularity of specializations of $\Ecal_1$ at $t = \theta^{q^k}, k \geq 1$} \label{evalE1thetapowers}

As an example of the evaluations $t = \theta, \theta^q,...$ of \S \ref{qthpowersevalsect}, we show that these functions may be explicitly determined in the case of the first vectorial Eisenstein series $\Ecal_1$ via Corollary \ref{e1totaufstar}.

Recall the definitions of $\bm{d}_1, \bm{d}_2$ and $\omega$ from \eqref{d1d2def} and \eqref{omegadef}, respectively. 
We deduce that 
\begin{align*}
\tau(\bm{d}_j)(z;\theta^{q^k}) = \left. \frac{ \sum_{i \geq 0} \frac{\alpha_i(z)^q}{\theta^{q^{i+1}} - t} z^{q^{i+1}(2-j)}}{\sum_{i \geq 0} \frac{\pitilde^{q^{i+1}}}{(\theta^{q^{i+1}} - t) D_i^q}} \right|_{t = \theta^{q^{k+1}}} = \left(\frac{D_k\alpha_{k}(z)  }{\pitilde^{q^{k}}}\right)^q z^{q^{k+1}(2-j)},
\end{align*}
for both $j = 1,2$ and all $k \geq 0$. 
Thus, by Corollary \ref{e1totaufstar} and Lemma \ref{chipropslem}, we obtain 
\[ \Ecal_1|_{t = \theta^{q^{k+1}}}(z) = - \pitilde h(z) \left(\frac{D_k\alpha_{k}(z)  }{\pitilde^{q^{k}}}\right)^q \left( \begin{matrix} -1 \\ z^{q^{k+1}}  \end{matrix}  \right). \]
In particular, we deduce that the $q$-th power of each normalized (\footnote{Normalized: the first non-zero coefficient of its $u$-expansion equals 1.}) para-Eisenstein series is a ratio of two single cuspidal Hecke eigenforms with $A$-expansion
\begin{align}
\left(\frac{D_k\alpha_{k}(z)  }{\pitilde^{q^{k}}}\right)^q = \frac{\sum_{a \in A_+} a^{q^{k+1}}u(az)}{\sum_{a \in A_+} a^q u(az)}.
\end{align}

\subsubsection{The connection of VMF with Petrov's special family}
Petrov, building on the work of B. L\'opez \cite{BLAdM10}, discovered in his thesis (see \cite{PAjnt} for the published version) a family of Drinfeld modular forms with special expansions at the cusp at infinity, similar to those discovered by Goss for his Eisenstein series \cite[(6.3)]{EGinv}, which Petrov dubbed {\em $A$-expansions} (\footnote{The terminology $A$-expansion refers to the fact that the expansion at the cusp at infinity determining these forms is given over the monic elements of $A$, rather than over the positive integers as originally required by the $u$-expansions of Goss. }). 
He has shown that the functions represented near infinity by the series
\[f_s(z) := \sum_{a \in A_+} a^{1+s(q-1)} u(az) \] are single-cuspidal elements of $M_{2+s(q-1)}^1(\GL_2(A))$, for all positive integers $s$. 
By an argument of Gekeler \cite[Theorem 3.1]{BLadm11}, such $A$-expansions are uniquely determined by the coefficients of the functions $\{u(az)\}_{a \in A_+}$, and their interest comes in part due to their good behavior with respect to Hecke operators \cite{PAjnt}.

\begin{theorem}
For all $d \geq 1$, letting $s = \frac{q^d - 1}{q-1}$ we have
\[ [\Ecal_1]_1|_{t = \theta^{q^d}} = -\pitilde f_s.\]
\end{theorem}
\begin{proof}
Such evaluations are possible by the analytic continuation which follows from Corollary \ref{e1totaufstar}, and their modularity follows from Proposition \ref{quasimodprop}. Finally, apply Proposition \ref{AexpE1prop}.
\end{proof}

\begin{remark}
The previous theorem also gives analytic continuation for all $z \in \Omega$ to the $A$-expansions, which at first only converge in some rigid analytic neighborhood of the infinity cusp, for these $f_s$.
The rigid analytic extension of $f_s$ to all of $\Omega$ was one of the trickier parts of Petrov's work. 
We expect to obtain all of Petrov's forms $f_s$ defined above from Eisenstein series of weight 1 via symmetric powers of the representation $\rho_t^*$ after evaluation of the variable $t$ at powers $\theta^{q^j}$. The authors hope to work out the details in a future work. 
We note that it is not nearly as straightforward to evaluate the higher weight Eisenstein series $\Ecal_k$ at the points $t = \theta^{q^j}$ when $q^j > k$, but that more forms with $A$-expansions may be obtained from those in his special family through hyperdifferentiation in the variable $z$; see \cite{APhyp} where this is carried out. 
\end{remark}

\subsection{$A$-expansions for vectorial Eisenstein series} \label{expansions forvectorialEisenstein}
Now we calculate vectorial $A$-expansions for the non-zero Eisenstein series $\Ecal_k$. We use the formalism of hyperderivatives in the variable $z$ to expedite this task. The $A$-expansions obtained in this section are useful in the proof below that the vectorial Eisenstein series are eigenforms for the Hecke operators defined in the previous section.

\subsubsection{Hyperderivatives in $z$} 
For $f \in \Thol(\Omega,\TT)$ and $z \in \CC_\infty$, we define the family $\{\hyp_{z}^{(n)} f, n \geq 0 \}$ of hyperderivatives of $f$ at $z$ via
\[f(z+\epsilon) = \sum_{n \geq 0} (\hyp_{z}^{(n)} f)(z) \epsilon^n, \]
where $\epsilon \in \CC_\infty$ is taken sufficiently small for the formula above to make sense. This definition gives rise to a family $\{\hyp_{z}^{(n)}, n \geq 0\}$ of hyperdifferential operators on $\Thol(\Omega, \TT)$ as follows from the work done in \cite{SUmathann}.

We shall use two properties of these hyperderivatives. First, they satisfy a \emph{Leibniz rule}. That is, for all $f,g \in \Thol(\Omega,\TT)$ we have
\begin{equation} \label{leibniz}
\hyp_{z}^{(n)}(fg)= \sum_{k = 0}^n (\hyp_{z}^{(k)} f) (\hyp_{z}^{(n-k)}g).
\end{equation}
Second, for all non-negative integers $k$,
\begin{equation} \label{reciphyp}
\hyp_{z}^{(n)} \frac{1}{(z+x)} = \frac{(-1)^n}{ (z+x)^{n+1} }.
\end{equation}

We use this formalism to prove the following generalization of Proposition \ref{AexpE1prop}. We recall that the functions $E^{(k)}$ are Goss' Eisenstein series in Gekeler's notation \cite[(5.9)]{EGinv}.

\begin{theorem} \label{eisAexp}
For positive integers $k \equiv 1 \pmod{q-1}$, let 
\[\lambda_k := -\frac{L(\chi_t,k)}{\pitilde^k} + \frac{1}{\omega(t)} \sum_{l = 0}^{\lfloor \log_q(k-1) \rfloor} \frac{ \zeta(k-q^l) + E^{(k-q^l)}(z)}{\pitilde^{k-q^l} (\theta^{q^l} - t) D_l},\] 
where $\lfloor \cdot \rfloor$ denotes the floor function, and the sum is understood to be empty when $k = 1$. For all such $k$, we have
\begin{flalign*}
\pitilde^{-k}\Ecal_k(z) =& {0 \choose 1} \lambda_k + \sum_{a \in A_+} {- \chi_t(a) \choose \chi_t(az)} G_k(u(az)).
\end{flalign*}
\end{theorem}

\begin{remark}
NB. that $\lambda_k \in \EE[[u]]$, so that after applying Lemma \ref{AGFfe}, this can be made to represent an expansion for $\Ecal_k \in \Theta_t \TT[[u]]^{2\times 1}$.
\end{remark}

\begin{proof}
Write $\Ecal_k := { e_1^k \choose e_2^k}$. We have already handled the case where $k = 1$ in Proposition \ref{AexpE1prop}, and we now assume $k > 1$. 

It was proved in \cite{FPannals}, that $e_1^k(z) = -\pitilde^k \sum_{a \in A_+} \chi_t(a) G_k(u(az))$. Now let us examine the second coordinate $e-2^{k}(z)$.

We need the following preliminary observations. From \eqref{reciphyp}, the assumption $k \equiv 1 \mod (q-1)$, and the identity 
\begin{equation} 
\sum_{d \in A} \frac{\chi_t(d)}{z+d} = -\pitilde u(z) \chi_t(z),
\end{equation}
proved in \cite[Theorem 1.1]{RPmathz}, we obtain
\begin{equation} \label{hyp1}
\sum_{d \in A} \frac{\chi_t(d)}{(z+d)^k} = (-1)^{k-1} \hyp_{z}^{(k-1)} \left( \sum_{d \in A} \frac{\chi_t(d)}{z+d} \right) = (-1)^k \pitilde \hyp_{z}^{(k-1)}(u \chi_t) (z).
\end{equation}
We shall also use
\begin{align} \label{hyp2} 
(\hyp_{z}^{(k-1)} u)(cz) &= (-\pitilde)^{k-1} G_k(u(cz)), \\
\label{chihyp} \hyp_z^{q^l} \chi_t &= \frac{\pitilde^{q^l}}{(\theta^{q^l - t})D_j}, \forall l \geq 0 \text{ and } \hyp_z^{q^l} \chi_t = 0, \text{ otherwise}.
\end{align}

Hence, removing the constant term via $e_2^k(z) + \sum_{d \in A_+} \frac{\chi_t(d)}{d^k} $ we obtain
\begin{align*}
 \sideset{}{'}\sum_{c \in A} \sum_{d \in A} \chi_t(d) (cz+d)^{-k} &= \\
&= -(-1)^k\pitilde \sum_{c \in A_+} (\hyp_{z}^{(k-1)} u \chi_t)(cz)
\end{align*}
\begin{align*}
=-(-1)^k\pitilde \sum_{c \in A_+} \sum_{j = 0}^{k-1} (\hyp_{z}^{(k-1-j)} u)(cz) (\hyp_{z}^{(j)}\chi_t)(cz) \\
= -\frac{(-1)^k\pitilde}{\omega(t)} \sum_{c \in A_+} \sum_{l = 0}^{\lfloor \log_q(k-1) \rfloor} (\hyp_{z}^{(k-1-q^l)} u)(cz) \frac{\pitilde^{q^l}}{(\theta^{q^l} - t) D_l} \\
 -{(-1)^k\pitilde}\sum_{c \in A_+}(\hyp_{z}^{(k-1)}u)(cz)\chi_t(cz) \\
= \frac{1}{\omega(t)} \sum_{l = 0}^{\lfloor \log_q(k-1) \rfloor} \frac{\pitilde^{q^l}}{(\theta^{q^l} - t) D_l} \left(-\pitilde^{k-q^l} \sum_{c \in A_+} G_{k-q^l}(u(cz))\right) \\
 + \pitilde^k \sum_{c \in A_+} G_k(u(cz))\chi_t(cz) \\
= \frac{1}{\omega(t)} \sum_{l = 0}^{\lfloor \log_q(k-1) \rfloor} \frac{\pitilde^{q^l}}{(\theta^{q^l} - t) D_l}\left( \sum_{a \in A_+} \frac{1}{a^{k-q^l}} + E^{(k-q^l)}(z) \right) \\
 + \pitilde^k\sum_{c \in A_+} G_k(u(cz))\chi_t(cz).
\end{align*}

From the first line to the second we have collected the elements of $A$ according to their leading coefficient, using that $k \equiv 1 \mod (q-1)$, and we have used the description of the second sum on the right side in terms of the hyperdifferential operator $\hyp_{z}^{(k-1)}$, as in \eqref{hyp1} above. 
From the second to the third lines, we have used the Leibniz rule \eqref{leibniz}.
From the third to the fourth, we use \eqref{chihyp}.
From the fourth to the fifth lines we have used \eqref{hyp2} above. Finally, from the fifth to the sixth lines we have used \cite[(6.3)]{EGinv}.
This concludes the calculation. 
\end{proof}

\subsection{Determinant maps and Ramanujan-Serre derivatives}
\label{Ramanujan-Serre}
We quickly digress to remind the reader of a determinant map that exists between weak VMF and was already considered in \cite{FPtaurecur, FPannals} in the construction of various important Drinfeld modular forms of full level. We use these determinant maps and the $A$-expansions obtained above for the vectorial Eisenstein series to make a connection with the Ramanujan-Serre derivatives introduced by Gekeler in the Drinfeld modular setting.

Given two VMF$^!$ $\Hcal_1,\Hcal_2$, let $[\Hcal_1, \Hcal_2]$ denote the square matrix whose first and second columns contain the entries of $\Hcal_1$ and $\Hcal_2$, respectively. The following result is immediate. 

\begin{proposition}
Let $\Hcal_1 \in \MM_{k_1}^{m_1}(\rho_t^*)^!$ and $\Hcal_2 \in \MM_{k_2}^{m_2}(\rho_t^*)^!$. We have
\[\det[\Hcal_1 \ \Hcal_2] \in \MM_{k_1+k_2}^{m_1+m_2+1}(\bm{1})^!.\]
\end{proposition}

For example, the forms $\det[\Ecal_1,\Ecal_k] \in \MM_{k+1}^1(\bm{1})$, where $k \geq q$, give representatives for all of the  single and not double cuspidal Drinfeld modular forms of full level. 

We have not yet been able to understand the behavior of the Hecke operators under this determinant map. In particular, we do not know if the forms $\det[\Ecal_1 \ \Ecal_k]$ are Hecke eigenforms.
 
\subsubsection{Relations with Ramanujan-Serre derivative} \label{RamSerDer}
In \cite[\S 8]{EGinv}, Gekeler defines a family of derivations $M_k^m(\GL_2(A)) \rightarrow M_{k+2}^{m+1}(\GL_2(A))$ given by
\[ \partial_k f := \frac{1}{\pitilde} \hyp_{z}^{(1)} f + k Ef,\]
where $E(z) := \sum_{a \in A_+} a u(az)$ is the false Eisenstein series of weight 2, represented here by its $A$-expansion for the cusp at infinity. 
These derivations are analogous to the Ramanujan-Serre derivative classically. 

It is evident from the $A$-expansion for $\Ecal_1$ that the first coordinate of $-\pitilde^{-1}\Ecal_1$ specializes at $t = \theta$ to $E$. We wish to relate the linear map 
\[\ev_\theta\det[\Ecal_1 \ \cdot \ ]:\MM_k^m(\rho_t^*) \rightarrow M_{k+1}^{m+1} \text{ given by } \Hcal \mapsto \ev_\theta\det[\Ecal_1 \ \Hcal]\] 
to Gekeler's $\partial$ when $\Hcal = \Ecal_k$. 


\begin{proposition}
For all integers $k > 1$ such that $k \equiv 1 \pmod{q-1}$ we have
\[\ev_\theta \det[ \Ecal_1, (k-1)\Ecal_k ] = -\pitilde \partial_{k-1} E^{(k-1)}.\]
\end{proposition}
\begin{proof}
From Proposition \ref{AexpE1prop}, we obtain $(\ev_\theta \Ecal_1)(z) = {-\pitilde E(z) \choose -1 + \pitilde z E(z)}$. 
By Theorem \ref{eisAexp} above, for all $k > 1$ such that $k \equiv 1 \pmod{q-1}$, we have
\[\ev_\theta \Ecal_k(z) = {-1 \choose z} \pitilde^{k}\sum_{a \in A_+} a G_k(u(az)) + {0 \choose 1}E^{(k-1)}(z).\]
Further, observe that by \cite[(3.4)(vii), (6.3), \S 8]{EGinv} 
\begin{align} \label{preQkeq}
\hyp_{z}^{(1)} E^{(k-1)}(z) = (k-1)\pitilde^{k}\sum_{a \in A_+} a G_k(u(az)).\end{align}
Thus, 
\[(k-1)\ev_\theta \Ecal_k = {-1 \choose z} \hyp_{z}^{(1)} E^{(k-1)} + {0 \choose 1}(k-1)E^{(k-1)},\]
which gives
\begin{align*}
\ev_\theta \det[\Ecal_1 \ (k-1) \Ecal_k] = \\
= -\pitilde E(z \hyp_{z}^{(1)} E^{(k-1)} + (k-1)E^{(k-1)}) - (z\pitilde E - 1)(-\hyp_{z}^{(1)} E^{(k-1)}) \\
 = -(k-1)\pitilde E E^{(k-1)} - \hyp_{z}^{(1)} E^{(k-1)} \\
 = -\pitilde \partial_{k-1} E^{(k-1)}.
\end{align*}
\end{proof}

\begin{remark}
We notice that when the characteristic of $K$ divides $k-1$, one has
\[\partial_k f = \frac{1}{\pitilde}\hyp_t^{(1)}f,\]
and one is led to ask about the modularity of $Q_k := \sum_{a \in A_+} a G_k(u(az))$ which one pulls from
\[\hyp_{z}^{(1)} E^{(k-1)}(z) = (k-1)\pitilde^{k}\sum_{a \in A_+} a G_k(u(az)),\]
for such $k$. By the proof of Proposition \ref{quasimodprop}, the function $Q_k$ is quasimodular, and not modular. 
\end{remark}


\section{Interpolation of Drinfeld modular forms of prime power levels}
\label{InterpolationofDrinfeldmodularforms}
One of the most intriguing features of the VMF studied in this note is that their coordinate functions specialize to Drinfeld modular forms of prime level $\pfrak$ upon making the replacement $t = \zeta$ for a root $\zeta \in \FF_q^{ac} \subset \CC_\infty$ of $\pfrak$. By introducing hyperderivatives in the variable $t$, we shall be able to show that such hyperderivatives of the coordinate functions of VMF specialize at $t = \zeta$ to forms of prime power levels. 
First we set up some preliminaries concerning Drinfeld modular forms. 

\subsection{Basic congruence subgroups} 
Throughout this section $\mfrak \in A$  denotes an arbitrary monic polynomial in $A$, although later we are going to make certain restrictions suitable with our purposes.

We recall that $\Gamma(\mfrak)$, \emph{the principal congruence subgroup of level $\mfrak$ in $\Gamma(1) := \GL_2(A)$}, is defined via the exact sequence
\[ 1 \rightarrow \Gamma(\mfrak) \rightarrow \Gamma(1) \rightarrow \GL_2(A/\mfrak A),\]
where the third arrow is the reduction of matrix coefficients modulo $\mfrak$; as Gekeler points out \cite[(3.5)]{EGjnt01}, this last arrow in the exact sequence above does not surject but lands in the subgroup of matrices with determinants in $\FF_q^\times$. 

We also have the \emph{Hecke congruence subgroup of level} $\mfrak$ given by
\[\Gamma_0(\mfrak) := \left\{ \gamma \in \Gamma(1) : \gamma \equiv \left( \begin{smallmatrix} * & * \\ 0 & *  \end{smallmatrix}  \right) \pmod{\mfrak} \right\}.\]
On $\Gamma_0(\mfrak)$ we have the character 
\[\eta : \Gamma_0(\mfrak) \rightarrow (A/\mfrak A)^\times\] 
given by 
\begin{equation}
\eta \left( \begin{smallmatrix} * & * \\ * & d  \end{smallmatrix}  \right) := d \pmod{\mfrak}. 
\end{equation}

We define $\Gamma_1(\mfrak)$ as the kernel of $\eta$. Explicitly,
\[\Gamma_1(\mfrak) := \left\{ \gamma \in \Gamma_0(\mfrak) : \gamma \equiv \left( \begin{smallmatrix} \xi & * \\ 0 & 1  \end{smallmatrix}  \right) \pmod{\mfrak}, \text{ for some } \xi \in \FF_q^\times \right\}.\]
The map $\eta$ surjects, and hence, $\Gamma_0(\mfrak)/\Gamma_1(\mfrak)$ is isomorphic to $(A/\mfrak A)^\times$. 

Similarly, the map
\[\Gamma_1(\mfrak)\xrightarrow{\beta} {A}/{\mfrak A}\]
defined by 
\begin{equation}
\beta \left( \begin{smallmatrix} * & b \\ * & *  \end{smallmatrix}  \right) := b \pmod{\mfrak}. 
\end{equation}
is a surjective group homomorphism with kernel $\Gamma(\mfrak)$ so that $\Gamma_1(\mfrak)/\Gamma(\mfrak)\cong A/\mfrak A$.
Similar results can be proved in the classical setting, see e.g. Diamond and Shurman's book \cite[\S 1.2]{DiShu}.

\subsubsection{Cusps for congruence subgroups}
The group $\GL_2(A)$ acts on the set $K \cup \{\infty\}$ via linear fractional transformations.  
Hence, given a subgroup $\Gamma'$ of $\GL_2(A)$, we can define a $\Gamma'$-equivalence relation
over $\mathbb{P}^1(K)=K \cup \{\infty\}$ in the following way: if $x,y\in \mathbb{P}^1(K)$,
then $x\sim_{\Gamma'} y$ if and only if there exists $\gamma\in \Gamma'$ such that $\gamma x=y$.
\begin{definition}
For each congruence subgroup $\Gamma(\mfrak) \subset \Gamma' \subset \GL_2(A)$, we call the $\Gamma'$-equivalence classes of points of $\mathbb{P}^1(K)$ the {\em cusps} of $\Gamma'$.
\end{definition}
All points of $K$ are $\Gamma'$-equivalent to $\infty$ if $\Gamma'=\GL_2(A)$. 
For the three groups $\Gamma(\mfrak)$, $\Gamma_1(\mfrak)$, and $\Gamma_0(\mfrak)$, with $\mfrak\neq 1$, Gekeler has given explicit descriptions of the cusps, including their number in \cite[\S 6]{EGjnt01}. In particular, since we will be dealing below with Drinfeld modular forms with character for $\Gamma_0(\pfrak)$, it is relevant to know the following result which follows from Gekeler's \cite[Proposition 6.7 (i)]{EGjnt01}.

\begin{lemma}
Let $\pfrak$ be a monic irreducible polynomial in $A$. There are $2$ cusps for $\Gamma_0(\pfrak)$ which may be represented by $0$ and $\infty$. \hfill \qed
\end{lemma}

\subsubsection{Petersson Slash Operators}
We recall a family of Petersson slash operators on rigid analytic functions $f :\Omega \rightarrow \CC_\infty$ defined for all non-negative integers $k,m$ and $\gamma \in \GL_2(K)$ by
\[(f|_k^m[\gamma])(z) := (\det\gamma)^m j(\gamma,z)^{-k} f(\gamma (z)).\]
When $k = m = 0$, we write more simply $f|[\gamma]$. 

One readily checks that we have $$(f|_k^m[\gamma_1])|_k^m[\gamma_2] = f|_k^m[\gamma_1\gamma_2]$$ for all $\gamma_1,\gamma_2 \in \GL_2(K)$, and non-negative integers $k,m$.

\subsubsection{Modular forms for congruence subgroups}   

The following definition is taken from \cite[(4.1)]{EGjnt01}. 

\begin{definition} \label{defDMF} 
Let $\Gamma(\mfrak) \subset \Gamma \subset \Gamma(1) := \GL_2(A)$ be any congruence subgroup.

A rigid analytic function $f :\Omega \rightarrow \CC_\infty$ shall be called \emph{modular} for $\Gamma$ of \emph{weight} $k$ and type $m \pmod{q-1}$ if the following two conditions are satisfied:
\begin{align}\label{Gammapmod}
& \text{for each } \gamma \in \Gamma, f|_k^m[\gamma] = f, \text{ and} \\ \label{0inftycusps}
& \text{for each } \gamma \in \Gamma(1), f|_k^m[\gamma] \text{ is bounded on } \{z \in \Omega : |z|_\Im \geq 1 \}.
\end{align}
We write $M_k^m(\Gamma)$ for the space of such functions, and we let
\[S_k^m(\Gamma) := \{ f \in M_k^m(\Gamma) : \forall \gamma \in \Gamma(1), f|_k^m[\gamma] \rightarrow 0 \text{ as } |z|_\Im \rightarrow \infty \}.\]
Functions in $S_k^m(\Gamma)$ are called {\it cuspidal} or {\it cusp forms} for $\Gamma$.
\end{definition}

The spaces $M_k^m(\Gamma)$ are finite dimensional, as they are subvector spaces of the finite dimensional space $M_k(\Gamma(\mfrak))$ (NB. that the type $m$ plays no role for the groups $\Gamma(\mfrak)$ when $\mfrak \neq 1$), whose exact dimension has even been calculated by Gekeler \cite[VII.6]{EGbook}.

\subsubsection{Expansions at the cusps} 
Let $\Gamma(\mfrak) \subset \Gamma \subset  \Gamma(1)$ be any congruence subgroup, as above. 

By {\it Goss' Lemma} \cite[Theorem 4.2]{DGbams}, for each $f \in M_k^m(\Gamma)$ and $\gamma \in \Gamma(1)$,  the function $f|_k^m[\gamma]$ has power series expansion in the uniformizer
\[u_\mfrak(z) := u(z/\mfrak) = \frac{\mfrak}{\pitilde} \sum_{a \in \mfrak A} \frac{1}{z-a} = \frac{1}{\ec(z/\mfrak)}.\]
with coefficients in $\CC_\infty$ which converges for all $z \in \Omega$ such that $|z|_\Im \gg 1$. 
We call this the {\it $u_\mfrak$-expansion for $f|_k^m[\gamma]$}.

Observe that this expansion determines $f$ uniquely since $\Omega$ is a connected rigid analytic space. In particular, we deduce the equivalence of Definition \ref{defDMF} above with the usual definition requiring a $u_\mfrak$-expansion at all cusps of $\Gamma$.

For the groups we are interested in below, namely $\Gamma_i(\mfrak)$, $i = 0,1$, the type $m$ plays a non-trivial role and the uniformizer $u_\mfrak$ (and not some power of it) is the proper uniformizer to use at all cusps for these groups.


\subsubsection{Example: Eisenstein series for principal congruence subgroups}
As a first basic example, we point out that for $\Gamma(\mfrak)$ the space of Eisenstein series has been explicitly described. We quote some properties contained, for example, in \cite{GCjnt}.

For $ \bm{v} \in (A/\mfrak A)^2\setminus\{(0,0)\}$, following Goss \cite{GossCM}, one may define
\[E^{(k)}_{\bm{v}} := \sum_{(a,b) \equiv \bm{v} \pmod{\mfrak}} \frac{1}{(az+b)^k}.\]
This Eisenstein series is a non-zero modular form of weight $k$ for $\Gamma(\mfrak)$. 

Cornelissen \cite[Proposition (1.12)]{GCjnt}
has shown that when the $\bm{v} \in (A/\mfrak A)^2\setminus\{(0,0)\}$ are restricted to a set $\mathcal{S}$ of representatives for the cusps of $\Gamma(\mfrak)$ the functions $E^{(k)}_{\bm{v}}$ are linearly independent
and span the complement in $M_k(\Gamma(\mfrak))$
of the subspace of cusp forms. 

It is also worth noting that, for each $z \in \Omega$, $1/E_{\bm v}^{(1)}(z)$ is an element of $\mathfrak{n}$-torsion for the Drinfeld module arising from the lattice $Az+A$. In particular, $E_{\bm v}^{(1)}(z) \neq 0$, for all $z \in \Omega$.


\subsubsection{Gekeler's false Eisenstein series}
Recall the {\it false Eisenstein series} $E$ of Gekeler, which is a \emph{Drinfeld quasi-modular form} in the sense of \cite{BvPfimrn} and is determined by its expansion at infinity
\[E = \sum_{a \in A_+} a \cdot u|[\alpha_a] \in A[[u]]; \]
NB. this would be $f_0$ in the notation introduced for Petrov's forms above. 
We obtain the analytic continuation to all $z \in \Omega$ of this form as well as its quasi-modularity from Corollary \ref{e1totaufstar} and Proposition \ref{quasimodprop}, since $E = [\Ecal_1]_1|_{t = \theta}$.

Since we have not seen it written elsewhere, we point out the following fact.

\begin{proposition}
For all monic irreducibles $\pfrak \in A$, 
\[E_\pfrak := E - E|_2^1[\alpha_\pfrak] \in M_2^1(\Gamma_0(\pfrak)).\]
Further,
\[E_\pfrak|_2^1[W_\pfrak] = -E_\pfrak.\]
\hfill \qed
\end{proposition}

\begin{remark}
From this result and the $A$-expansion for $E$, one easily obtains the $u_\pfrak$-expansions at infinity and at zero for the function $E_\pfrak$ and learns that this function gives a canonical representative for the one dimensional space of single cuspidal forms of weight two for $\Gamma_0(\pfrak)$. 
\end{remark}

\subsection{Modular forms for $\Gamma_0(\pfrak)$ with character via specialization} 
Now we turn to those Drinfeld modular forms which may be obtained from the individual coordinate functions of VMF via specialization of the variable $t$ in the roots of unity $\FF_q^{ac} \subset \CC_\infty$. 
We begin with a summary of certain results connected with the Carlitz module for motivation.

\subsubsection{Specialization for the Carlitz module} \label{GTsumssection}
Given an element $\phi = \sum_{i \geq 0} c_i t^i \in \TT$, we may specialize the variable $t$ at an element $\zeta$ in the algebraic closure $\FF_q^{ac} \subset \CC_\infty$ of $\FF_q$ to obtain an element $$\ev_\zeta(\phi)=\sum_{i\geq 0}c_i\zeta^i \in \CC_\infty.$$ If $\phi$ is the function $\omega$ of Anderson and Thakur, which arises in connection with the Carlitz module, we have the following instance of this --- discovered originally by Angl\`es and the first author \cite{APinv} and appearing with another proof in \cite{FPRP}.

Let $\zeta \in \FF_q^{ac}$ and $\pfrak \in A$ its minimal polynomial. Associated to the $\FF_q$-algebra map $\chi_\zeta : A \rightarrow \FF_q^{ac}$ determined by $\theta \mapsto \zeta$, we have the {\it basic Gauss-Thakur sum}
\[ \gfrak(\chi_\zeta) := \sum_{a \in (A/\pfrak A)^\times} \chi_\zeta(a)^{-1} \mathfrak{e_c}(\frac{a}{\pfrak}). \]
Angl\`es and the first author proved that for all $\zeta$, as above,
\begin{align} \label{APgtsumseq} \ev_\zeta(\omega) = \chi_\zeta(\pfrak') \gfrak(\chi_\zeta), \end{align}
where $\pfrak'$ denotes the formal derivative of $\pfrak$ with respect to $\theta$.

\begin{definition}
For a rigid analytic function $f : \Omega \rightarrow \TT$ and $\zeta \in \FF_q^{ac}$, we define
the associated {\it evaluation at $\zeta$} 
\[\ev_\zeta( f) : \Omega\rightarrow \CC_\infty,\] 
as the composition map $\Omega\xrightarrow{f} \TT \xrightarrow{\ev_\zeta}\CC_\infty$; $\ev_\zeta(f)$ is a rigid analytic function of the variable $z$ on $\Omega$.
More generally, for a vectorial function $\mathcal{H} : \Omega \rightarrow \TT^l$, we define $\ev_\zeta(\mathcal{H})$ by applying $\ev_\zeta$ on each coordinate. \end{definition}

The functions $\ev_\zeta(\psi_1)$ and $\ev_\zeta(\chi_t)$ are well-defined,
and satisfy interesting properties, as shown in \cite{FPRP}. For convenience, we recall in the next lemma
the main results obtained there, in this connection.

\begin{lemma} \label{Psievallem} 
For all $\zeta \in \FF_q^{ac}$, let $\chi = \chi_\zeta$, as above. We have
\begin{align} \pfrak\ev_\zeta(\chi_t)(\pfrak z) &\in A[\zeta,\mathfrak{e_c}(\pfrak^{-1})][\mathfrak{e_c}(z)], \\
\pfrak \ev_\zeta(\psi_{1})(\pfrak z) &\in A[\zeta, \mathfrak{e_c}(\pfrak^{-1})][[u(z)]], \text{ and} \\
\frac{\pfrak \ev_\zeta(\psi_{1})(\pfrak z)}{u(z)^{\frac{|\pfrak|}{q}(q-1)}} & \rightarrow (-1)^{\deg \pfrak +1} \gfrak(\chi^{-1}) \chi^{-1}(\pfrak')  \text{ as } |z|_{\Im} \rightarrow \infty. 
\end{align} \hfill \qed
\end{lemma}

\subsubsection{Modular forms with character for $\Gamma_0(\pfrak)$}
Let $\zeta \in \CC_\infty$ be a fixed root of the monic irreducible polynomial $\pfrak \in A$, i.e. $\pfrak(\zeta) = 0$, and define the character 
\[\eta_\zeta : \Gamma_0(\pfrak) \rightarrow \FF_q(\zeta)^\times \subset \CC_\infty^\times \text{ via} \]
\[ \left( \begin{matrix} a & b  \\ c & d \end{matrix}  \right) \mapsto d(\zeta). \] 
We have $\Gamma_1(\pfrak) = \ker \eta_\zeta$.

As $\Gamma_1(\pfrak)$ is a normal subgroup of $\Gamma_0(\pfrak)$, this latter group acts $\CC_\infty$-linearly on $M_k^m(\Gamma_1(\pfrak))$ via the Petersson slash operator by 
\[f \mapsto (f|_{k}^m{[\gamma]}), \ \forall \gamma \in \Gamma_0(\pfrak).\] 

\begin{definition}
Let $k$ be a positive integer and $m$ a residue $\pmod{q-1}$. For each $l = 0,1,\dots,|\pfrak|-2$, define
\[M_k^m(\pfrak,\eta_\zeta^l) := \{f \in M_k^m(\Gamma_1(\pfrak)) : f|_{k}^m [\gamma] = \eta_\zeta(\gamma)^{l} f \text{ for all } \gamma \in \Gamma_0(\pfrak)\}.\]
We call the functions in $M_k^m(\pfrak,\eta_\zeta^l)$ {\it Drinfeld modular forms of weight $k$, type $m$ and character $\eta_\zeta^l$}.
We may refer to these functions more loosely as {\it Drinfeld modular forms with character}.
\end{definition}


\begin{lemma}
1. If $M_k^m(\pfrak,\eta_\zeta^l) \neq 0$, necessarily $l+k \equiv 2m \pmod{q-1}$.

2. $M_k^m(\Gamma_0(\pfrak)) = M_k^m(\pfrak,\eta_\zeta^{0})$.

3. $M_k^m(\Gamma_1(\pfrak)) = \oplus_l M_k^{m}(\pfrak,\eta_\zeta^l)$.
\end{lemma}

\begin{proof}
The first comes in the usual way by consideration of 
\begin{equation} \label{sclrmattrickeq} \lambda^{2m-k}f = f|_k^m[\left(\begin{smallmatrix} \lambda & 0 \\ 0 & \lambda \end{smallmatrix} \right)] = \lambda^l f,\end{equation}
which holds for all $\lambda \in \FF_q^\times$.
The second follows since $\eta_\zeta^{0}$ is the trivial character for $\Gamma_0(\pfrak)/\Gamma_1(\pfrak)$.
The final claim follows from Maschke's Theorem \cite[Ch. XVIII, Theorem 1.2]{Lang}, since the cardinality $|(A/\pfrak A)^\times| = |\pfrak| - 1$ is coprime to the characteristic of $\CC_\infty$.
\end{proof}

\subsubsection{Holomorphy / expansions at cusps for Drinfeld modular forms with character}

Each function $f \in M_k^m(\pfrak, \eta_\zeta^l)$ is a modular form for $\Gamma_1(\pfrak)$, and hence has a $u_\pfrak$-expansion at the cusps of $\Gamma_1(\pfrak)$. By the transformation rule $f|_{k}^m [\gamma] = \eta_\zeta(\gamma)^{l} f$, which holds for all $\gamma \in \Gamma_0(\pfrak)$, one easily observes that it is enough to check holomorphy at the zero and infinity cusps, i.e. representatives for the cusps of $\Gamma_0(\pfrak)$. We observe that the matrix 
\[W_\pfrak := \left( \begin{smallmatrix} 0 & -1 \\ \pfrak & 0 \end{smallmatrix} \right) \in \GL_2(K) \]
is in the normalizer in $\GL_2(K)$ of $\Gamma_0(\pfrak)$,  
and the following simple consequence is easily checked; one may also consult \cite{RPmf16}. Observe that  $W_\pfrak$ sends the cusp at $\infty$ to the cusp at $0$ and vice-versa. 

\begin{lemma} \label{invollem}
For each $f \in M_k^m(\pfrak, \eta_\zeta^l)$, we have
\[f|_k^m[ W_\pfrak] \in M_k^{m-l}(\pfrak,\eta_\zeta^{-l}).\]
In particular,  both $f$ and $f|_k^m[W_\pfrak]$ have expansions in $\CC_\infty[[u]]$. \hfill \qed
\end{lemma}

\begin{remark}
Notice that in contrast to the case of modular forms for $\GL_2(A)$ and $\Gamma_0(\pfrak)$, for $\Gamma_1(\pfrak)$ we can have modular forms of the same weight yet with \emph{different} types! Indeed, we give an example in \S \ref{exEWCwt1}. 

The usual trick used \eqref{sclrmattrickeq} cannot be applied in this situation since the matrices $\left(\begin{smallmatrix} \lambda & 0 \\ 0 & \lambda \end{smallmatrix} \right)$, with $\lambda \in \FF_q^\times$, do not belong to $\Gamma_0(\pfrak)$, for any irreducible $\pfrak$. 
\end{remark}

\subsubsection{Forms with character via specialization}
Finally, we arrive at our first result connecting the specialized coordinate functions of $\TT$-valued VMF to Drinfeld modular forms with character. For all $\mfrak \in A_+$, we introduce the matrices
\[ \alpha_\mfrak := \left( \begin{matrix} \mfrak & 0 \\ 0 & 1 \end{matrix} \right) \in \GL_2(K).\]

\begin{proposition} \label{eigenevalprop}
For all $\Hcal = {h_1 \choose h_2} \in \MM_k^m(\rho_t^*)$, all $\zeta \in \FF_q^{ac} \subset \CC_\infty$ with minimal polynomial $\pfrak \in A$, and each $l = 0,1,\dots,\deg (\pfrak) - 1$, we have 
\begin{align} \label{evalpropeq1}
& \ev_{\zeta^{q^l}}(h_1) \in M_k^{m+1}(\pfrak, \eta_\zeta^{q^l}), \text{ and } \\
\label{evalpropeq2} & \ev_{\zeta^{q^l}}(h_1)|_k^m[W_\pfrak] = -\ev_{\zeta^{q^l}}(h_2)|_k^m[\alpha_\pfrak] \in M_k^{m}(\pfrak, \eta_\zeta^{-q^l}).\end{align} 
\end{proposition}

\begin{proof}
Equation \eqref{evalpropeq1} follows directly from the definition of a $\TT$-valued VMF, Lemma \ref{Psievallem}, and the equivalent conditions of Corollary \ref{infequiv1}. 
Indeed, from \eqref{modularity} we obtain
\[h_1(\gamma z) =  \frac{j(\gamma,z)^k}{(\det \gamma)^{m+1}} (d(t) h_1(z) -c(t) h_2(z)),\]
for all $\gamma \in \Gamma(1)$ and $z \in \Omega$. If $\gamma \in \Gamma_0(\pfrak)$, then $c \in \pfrak A$, and hence we obtain the desired modular transformation after evaluating at $t = \zeta^{q^l}$. By Corollary \ref{infequiv1}, $h_1 \in u\TT[[u]]$, and we may write
\[h_2 = h_3 - \chi_t h_1\]
for some $h_3 \in \TT[[u]]$. Thus, by Lemma \ref{Psievallem}, $(\ev_{\zeta^{q^l}}h_1)|_k^m[\gamma]$ is bounded on $\{|z|_\Im \geq 1\}$ for all $\gamma \in \Gamma(1)$.

The equality in \eqref{evalpropeq2} follows directly from \eqref{modularity}, and one sees that 
\[-\ev_{\zeta^{q^l}}(h_2)|_k^m[\alpha_\pfrak] \in M_k^{m}(\pfrak, \eta_\zeta^{|\pfrak|-1-q^l})\] 
either directly from the definition of a $\TT$-valued VMF or by appealing to Lemma \ref{invollem}.
\end{proof}


\subsection{Example: Eisenstein series of weight 1 with character} \label{exEWCwt1}

Recall the vectorial Eisenstein series of weight one considered in \S \ref{Eisseriessec}:
\[\Ecal_1(z) := {\epsilon_1(z) \choose \epsilon_2(z)} = \sideset{}{'}\sum_{a,b \in A} \frac{1}{az+b} {a(t) \choose b(t)}.\]
We fix $\zeta \in \FF_q^{ac} \subset \CC_\infty$ with minimal polynomial $\pfrak \in A_+$ and write $\epsilon_i^\zeta$ for $\ev_\zeta(\epsilon_i)$. 
By Proposition \ref{eigenevalprop} we have
\[\epsilon_1^\zeta \in M_1^1(\pfrak,\eta_\zeta) \ \text{ and } \]
\[\epsilon_2^\zeta|_1^1[\alpha_\pfrak] =  -\epsilon_1^\zeta|^1_1[W_\pfrak] \in  M_1^0(\pfrak,\eta_\zeta^{|\pfrak|-2}).\]  

Using use the $A$-expansion for $\Ecal_1$ given above in Proposition \ref{AexpE1prop}, we immediately obtain 
\begin{align} \label{Aexpepsilon1} \epsilon_1^\zeta(z) = -\pitilde \sum_{a \in A_+} a(\zeta) u(az) \in A[\zeta][[u]],
\end{align}
which is an $A$-expansion in the sense of Petrov, as for $f_s$.
Similarly, from the $A$-expansion for $\Ecal_1$ (Proposition \ref{AexpE1prop}), Lemma \ref{Psievallem}, and equation \eqref{APgtsumseq}, one may obtain some kind of series expansion for $\epsilon_2^\zeta$ indexed by the monics of $A$. We do not use it, and we refrain from writing it here.

From the expansion \eqref{Aexpepsilon1} for $\epsilon_1^\zeta$, which is valid in some neighborhood of the cusp at infinity, we see that this function is not identically zero. 
Thus, allowing the root $\zeta$ of $\pfrak$ to vary gives $2\deg \pfrak$ linearly independent (indeed, they lie in different eigenspaces of the $\Gamma_0(\pfrak)$ action), non-cuspidal(\footnote{We remind the reader, that Cornelissen proves \cite[Theorem (6.9.1)]{GCsurvey} --- using a proof which he attributes to Gekeler --- the surprising fact that there are no cusp forms of weight one for any congruence subgroup of $\GL_2(A)$.}) forms in the various spaces $M_1^m(\Gamma_1(\pfrak))$ after evaluation. 
In \cite{RPmf16}, using a classical approach, we are successful in constructing as many linearly independent Eisenstein series of weight one with character as there are cusps for $\Gamma_1(\pfrak)$, namely $2\frac{|\pfrak|-1}{q-1}$; see \cite[Proposition 6.6 (i)]{EGjnt01} for the number of cusps of $\Gamma_1(\mfrak)$ for general $1 \neq \mfrak \in A_+$.

By the discussion above, both $\epsilon_1^\zeta$ and $\epsilon_2^\zeta|[\alpha_\pfrak]$ are non-cuspidal modular forms of weight 1 for $\Gamma(\pfrak)$, and are thus expressible in the basis given by Cornelissen-Gekeler-Goss above. We have
\begin{align} \label{epsilonsincongeis}
\epsilon_1^\zeta = \sum_{0 \neq |c| < |\pfrak| } c(\zeta) \sum_{d \in A/\pfrak A} E_{(c,d)}^{(1)} \ \text{ and } \
\epsilon_2^\zeta|[\alpha_\pfrak] = \sum_{0 \neq |d| < |\pfrak| } d(\zeta) E_{(0,d)}^{(1)}.
\end{align} 
These are completely classical in shape. One easily takes the expasion for $\epsilon_1^\zeta$ in \eqref{epsilonsincongeis} to the $A$-expansion given in \eqref{Aexpepsilon1}. 
One can also obtain the following series expansion index by the monics of $A$ for $\epsilon_2^\zeta|[\alpha_\pfrak]$ from \eqref{epsilonsincongeis},
\begin{align} \label{epsilon2nice}
\epsilon_2^\zeta|[\alpha_\pfrak] = -L(\chi_t,1) + \frac{-\pitilde}{ \pfrak} \sum_{a \in A_+} \sum_{0 \neq |d| < |\pfrak|}  \frac{d(\zeta) \ec(d/\pfrak)^{q-2} u^{q-1}|[\alpha_a]}{\ec(d/\pfrak)^{q-1} u^{q-1}|[\alpha_a] + 1}.
\end{align}
This final expansion, which is not an $A$-expansion in the sense of Petrov, but something new, should be compared with those in \cite{RPmf16} wherein such examples are considered and explained.

\subsubsection{A family of congruences} 
We obtain the following immediate corollary to the $A$-expansion for $\epsilon_1^\zeta$ given in \eqref{Aexpepsilon1}. 

\begin{theorem}
For each $\zeta \in \FF_q^{ac}$ with minimal polynomial $\pfrak \in A$, there exists a form $f_\zeta \in M_1^1(\pfrak,\eta_\zeta) \cap A[\zeta][[u]]$ such that $E \equiv f_\zeta \pmod{(\theta - \zeta)}$. 
\end{theorem}
\begin{proof}
Indeed, just let $f_\zeta := \epsilon_1^\zeta$ and compare $A$-expansions.
\end{proof}

\begin{remark}
The previous result also gives an example of two forms for $\Gamma_1(\pfrak)$, with different weights, which have congruent $u$-expansions, namely $E_\pfrak$ and the $f_\zeta$ just defined.
\end{remark}

\subsubsection{$\vfrak$-adic modular forms}
We have seen above in Proposition \ref{quasimodprop} --- and more explicitly in \S \ref{evalE1thetapowers} --- that the forms $\frac{-1}{\pitilde} \ev_{\theta^{q^d}}[\Ecal_1]_1$ are modular for all $d \geq 1$ and have $u$-expansions in $A[[u]]$. Now we look at their $\pfrak$-adic convergence for $\zeta \in \FF_q^{ac}$ with minimal polynomial $\pfrak$. These notions were first investigated for Drinfeld modular forms by C. Vincent \cite{CVjnt} and D. Goss \cite{DGvadic}.

First, fix $\zeta \in \FF_q^{ac}$ with minimal polynomial $\pfrak$, as above, and embed $\FF_q[\zeta]$ in the $\pfrak$-adic completion $A_\pfrak$ of $A$ so that $\theta - \zeta$ has $\pfrak$-adic valuation $1$. 

\begin{definition}
We shall say that a modular form $f$ with $u$-expansion at infinity in $A[\zeta][[u]] \subset A_\pfrak[[u]]$ is a $\pfrak$-adic modular form for $\GL_2(A)$, if there exists a sequence of modular forms $\{f_n\}_{n \gg 0} \subset M(\GL_2(A))$, with $u$-expansion at infinity in $A[[u]]$, such that the smallest $\pfrak$-adic valuation of all of the coefficients of $f - f_n$ tends to infinity with $n$.
\end{definition}

\begin{proposition}
For each $\zeta \in \FF_q^{ac}$ with minimal polynomial $\pfrak \in A$ of degree $d$, the form 
\[\frac{-1}{\pitilde}\ev_\zeta[\Ecal_1]_1 \in A[\zeta][[u]] \cap M_1^1(\Gamma_1(\pfrak))\] 
is a $\pfrak$-adic modular form for $M(\GL_2(A))$. 
\end{proposition}
\begin{proof}
Let $\zeta$ and $\pfrak$ be as in the statement; so, $d = \deg \pfrak$. We have $A$-expansions for the forms $\frac{-1}{\pitilde}\ev_\zeta[\Ecal_1]_1$ and $\frac{-1}{\pitilde} \ev_{\theta^{q^{nd}}}[\Ecal_1]_1$, and we examine their difference. We have 
\begin{flalign*} \frac{-1}{\pitilde}\ev_{\theta^{q^{nd}}}[\Ecal_1]_1(z) + \frac{1}{\pitilde} \ev_\zeta[\Ecal_1]_1(z) &= \sum_{a \in A_+} (a^{q^{nd}} - a(\zeta)) u(az) \\
& = \sum_{a \in A_+} (a - a(\zeta))^{q^{nd}} u(az),
\end{flalign*}
and $\theta - \zeta$ divides $a - a(\zeta)$, for all $a \in A$. Hence $(\theta - \zeta)^{q^{nd}}$ divides the right side above, which in turn implies that it divides each $u$-expansion coefficient of the difference of coordinate functions above and finishes the proof.
\end{proof}


\subsection{Prime power levels via hyperdifferentiation and specialization} 
Now we consider how VMF give rise to Drinfeld modular forms for the congruence subgroups $\Gamma_1(\pfrak^n)$, for monic irreducibles $\pfrak \in A$. 

\subsubsection{Hyperderivatives in $t$}

We define a family of higher derivations or hyperderivatives $\hyp_{t}^{(n)} : \TT \rightarrow \TT$ in the variable $t$ via
\[\phi(t+\epsilon) = \sum_{n \geq 0} (\hyp_{t}^{(n)} \phi)(t) \epsilon^n \in \TT[[\epsilon]]. \]
We extend this to rigid analytic functions $f: \Omega \rightarrow \TT$ in the obvious way and observe that $\hyp_t^{(n)} f : \Omega \rightarrow \TT$ is again a rigid analytic function. 

The main property we shall use of this family is that each member satisfies a Leibniz rule: for rigid analytic $f,g :\Omega \rightarrow \TT$ and all positive integers $n$, we have
\[\hyp_{t}^{(n)}(fg) = \sum_{j  = 0}^n \hyp_{t}^{(j)} f \cdot \hyp_{t}^{(n-j)}g .\]
We extend the family $\hyp_{t}^{(n)}$ to matrices. For $(a_{ij}) \in \Mat_{m,n}(\Hol(\Omega, \TT))$, let 
\[\hyp_{t}^{(n)}(a_{ij}) := \left( \hyp_{t}^{(n)} a_{ij} \right).\] 

We require the following facts. In all, $\varpi \in \FF_q[t]$ is an arbitrary non-constant polynomial. 

\begin{lemma} \label{hypdivislem}
Let $n$ be a positive integer, and suppose $\alpha \in \FF_q + \varpi^n \FF_q[t]$, then, for all $j = 1,..,n-1$,
\[\varpi \text{ divides } \hyp_{t}^{(j)}(\alpha) \text{ in } \FF_q[t].\]
\end{lemma}

\begin{proof}
Let $\alpha = \xi + \varpi^n \beta \in \FF_q + \varpi^n \FF_q[t]$. Using the Leibniz rule and linearity, we have
\[\hyp_{t}^{(j)} \alpha =  \sum_{k = 0}^j (\hyp_{t}^{(k)}\varpi^n) (\hyp_{t}^{(j-k)}\beta).\]
Thus, it suffices to show: for all $j = 0, 1, \dots, n-1$, 
\[ \varpi \text{ divides } \hyp_{t}^{(j)}(\varpi^n) \text{ in } \FF_q[t].\]
We prove this by strong-induction on $n$. The result is clear for $n = 1$. Let $j \in \{0,1,\dots,n-1\}$. By the Leibniz rule, we have
\[\hyp_{t}^{(j)} (\varpi^n) = \varpi \hyp_{t}^{(j)} (\varpi^{n-1}) + \sum_{k = 1}^j \hyp_{t}^{(k)}(\varpi)\hyp_{t}^{(j-k)}(\varpi^{n-1}).\]
By the induction hypothesis, $\varpi$ divides $\hyp_{t}^{(j-k)} (\varpi^{n-1})$ for all $k \in \{1,2,\dots,j\}$. Thus $\varpi$ divides the right side above, and we are done.
\end{proof}

\begin{lemma} \label{chihypers}
Let $\zeta \in \FF_q^{ac}$ with minimal polynomial $\pfrak \in A$. For all positive integers $j$, $\ev_\zeta (\hyp_{t}^{(j-1)} \chi_t)$ is $\pfrak^j A$-periodic and we have that 
\[ u(z)(\hyp_{t}^{(j-1)}\chi_t)(z) \rightarrow 0,  \text{ as } |z|_\Im \rightarrow \infty.\] 
\end{lemma}

\begin{proof}
From $\chi_t(z+a) = \chi_t(z) + a(t)$, we obtain 
\[(\hyp_{t}^{(j-1)}\chi_t)(z+a) = (\hyp_{t}^{(j-1)}\chi_t)(z) + \hyp_{t}^{(j-1)}a(t).\] 
Thus if $a \in \pfrak^j A$, the previous corollary gives the periodicity after evaluation at $t = \zeta$. 

For the second claim, from $\psi_1 = u \chi_t$ we obtain 
\[ \pitilde u(z) (\hyp_{t}^{(j-1)}\chi_t)(z) = \sum_{a \in A} \frac{\hyp_{t}^{(j-1)}a(t)}{z-a},\] 
and the right side clearly vanishes at infinity since $||\hyp_{t}^{(j-1)}a(t)|| \leq 1$, for all $j \geq 1$ and $a \in A$. 
\end{proof}

\subsubsection{Modular forms for $\Gamma_1(\pfrak^n)$}

\begin{lemma} \label{modularleibniz}
Let $\Hcal = {h_1 \choose h_2} \in \MM_k^m(\rho_t^*)^!$. For all non-negative integers $n$, we have 
\[(\hyp_{t}^{(n)} \Hcal)(\gamma (z)) = \frac{j(\gamma, z)^k}{(\det\gamma)^m} \sum_{j = 0}^n (\hyp_{t}^{(j)} \rho_t^*(\gamma)) (\hyp_{t}^{(n-j)}\Hcal)(z). \]
\end{lemma}

\begin{proof}
Let $\gamma = \left(\begin{smallmatrix} a & b \\ c & d \end{smallmatrix} \right)$, so that $\rho_t^*\gamma = \frac{1}{ad-bc}\left(\begin{smallmatrix} d(t) & -c(t) \\ -b(t) & a(t) \end{smallmatrix} \right)$. We have 
\[h_1(\gamma (z)) = \frac{j(\gamma,z)^k}{(\det\gamma)^{m+1}}(d(t) h_1(z) -c(t) h_2(z)).\] 
Applying the $\hyp_{t}^{(n)} $-difference operator, and using the Leibniz rule, we see that 
\[(\hyp_{t}^{(n)} h_1)(\gamma (z)) \text{ equals}\]
\[ \frac{j(\gamma,z)^k}{(\det\gamma)^{m+1}}\sum_{j = 0}^n \left( (\hyp_{t}^{(j)} d(t)) (\hyp_{t}^{(n-j)} h_1)(z) -(\hyp_{t}^{(j)} c(t)) (\hyp_{t}^{(n-j)} h_2)(z)\right). \]
Similarly for $h_2$. Putting everything together finishes the proof.
\end{proof}

\begin{proposition} \label{hyperdiffprop}
Let $\Hcal = {h_1 \choose h_2} \in \MM_k^m(\rho_t^*)$, $\zeta \in \FF_q^{ac}$ with minimal polynomial $\pfrak$, and $n$ a positive integer. We have 
\[\ev_\zeta(\hyp_{t}^{(n-1)}h_1) \in M_k^{m+1}(\Gamma_1(\pfrak^n)) \setminus M_k^{m+1}(\Gamma_1(\pfrak^{n-1})).\]
\end{proposition}

\begin{proof}
Let $\gamma = \left(\begin{smallmatrix} a & b \\ c & d \end{smallmatrix} \right) \in \GL_2(A)$. From Lemma \ref{modularleibniz}, we see that 
\begin{align}\label{hypertrans}
(\hyp_{t}^{(n-1)}h_1)(\gamma (z)) = \frac{j(\gamma,z)^k}{(\det\gamma)^{m+1}} \sum_{j = 0}^{n-1} \left( \begin{smallmatrix} \hyp_{t}^{(j)}d(t), -\hyp_{t}^{(j)}c(t) \end{smallmatrix}\right) \cdot \left( \begin{smallmatrix} \hyp_{t}^{(n-1-j)}h_1(z) \\ \hyp_{t}^{(n-1-j)}h_2(z) \end{smallmatrix}\right).
\end{align}
If additionally, $\gamma \in \Gamma_1(\pfrak^n)$, then $\pfrak(t)$ divides $c(t)$ and both $\hyp_{t}^{(j)}d(t), \hyp_{t}^{(j)}c(t)$ for $j = 1,2,\dots,n-1$, as shown above. Thus evaluating at $t = \zeta$, we obtain
\[(\ev_\zeta \hyp_{t}^{(n-1)}h_1)(\gamma (z)) = \frac{j(\gamma,z)^k}{(\det\gamma)^{m+1}} (\ev_\zeta \hyp_{t}^{(n-1)}h_1)(z),\]
for all $\gamma \in \Gamma_1(\pfrak^n)$.
Note that had $\gamma$ been in $\Gamma_1(\pfrak^j)$ for some $j<n$, we would not have $\pfrak(t) | \hyp_{t}^{(n-1)}c(t)$ in general, and so the functional equation in the line above does not hold in general, showing that $\ev_\zeta (\hyp_{t}^{(n-1)}h_1)$ does not lie in $M_k^{m+1}(\Gamma_1(\pfrak^j))$ for any $j < n$.

For holomorphy at the cusps, it suffices by \eqref{hypertrans}, which holds for all $\gamma \in \GL_2(A)$, to know that $\ev_\zeta \hyp_{t}^{(j)} h_i$ are holomorphic at infinity for $j =0,1,\dots, n-1$ and $i = 1,2$, and this follows from the relation $h_2(z) = h_3(z) -\chi_t(z)h_1(z)$, with $h_3 \in \TT[[u]]$ and $h_1 \in u\TT[[u]]$ (i.e Cor. \ref{infequiv1}), the Leibniz rule, and Lem. \ref{chihypers}.
\end{proof}

\subsubsection{A non-classical family of vectorial Drinfeld modular forms for $\Gamma_1(\pfrak^n)$}

We are lead to consider the following family $\phi_t^{(n)}$ of faithful $\FF_q$-algebra representations:
\[\FF_q[t] \xrightarrow{\phi_t^{(n)}} \Mat_n(\FF_q[t]) \ \text{ given by } \ a \xmapsto{\phi_t^{(n)}} \left( \begin{matrix} a & \hyp_t^{(1)}a & \cdots & \hyp_t^{(n-1)}a \\ 0 & a & \cdots & \hyp_t^{(n-2)}a \\ \vdots & \ddots & \ddots & \vdots \\ 0 & \cdots & 0 & a \end{matrix} \right).\]
We point out that the action of $A$ via $\phi_t^{(n)}$ corresponds exactly to the action of the $n$-th tensor power of the Carlitz module $C^{\otimes n}$ on the tangent space $\text{Lie}(C^{\otimes n})$.

\begin{lemma}
Let $\zeta \in \FF_q^{ac} \subset \CC_\infty$ with minimal polynomial $\pfrak \in A_+$. The image of the composition
\[\phi_\zeta^{(n)} : \FF_q[t] \xrightarrow{\phi^{(n)}_t} \Mat_n(\FF_q[t]) \xrightarrow{\ev_\zeta} \Mat_n(\FF_q(\zeta))\]
is isomorphic to $A/\pfrak^n A$. In particular, \medskip

$\left\{  \left( \begin{smallmatrix} a(\zeta) & (\hyp_t^{(1)}a)(\zeta) & \hdots & (\hyp_t^{(n-1)}a)(\zeta) \\ 0 & a(\zeta) & \hdots & (\hyp_t^{(n-2)}a)(\zeta) \\ \vdots & \ddots & \ddots & \vdots \\ 0 & \hdots & 0 & a(\zeta) \end{smallmatrix} \right) : |a|<|\pfrak^n| \textit{ and } a(\zeta) \neq 0 \right\} \cong (A/\pfrak^n A)^\times. $
Further, the map
\[\eta_\zeta^{(n)} : \Gamma_0(\pfrak^n) \rightarrow \GL_n(\FF_q(\zeta)) \subset \GL_n(\CC_\infty) \text{ given by } \left( \begin{matrix} a & b \\ c & d \end{matrix} \right) \mapsto \phi_\zeta^{(n)}(\chi_t(d)) \]
is a group homomorphism with kernel $\Gamma_1(\pfrak^n)$.
\end{lemma}

\begin{proof}
The composite map $\phi^{(n)}_\zeta$ surjects on its image, and for an element to be in the kernel, all entries of the matrix in its image must be zero. By the lemmas above, this happens if and only if $a \in \pfrak^n A$. Thus the composite map factors through $A/\pfrak^n A$, and the invertible elements in the image are exactly those matrices whose diagonal entries are non-zero. This gives exactly the condition $a(\zeta) \neq 0$. The final claim follows directly from Lemma \ref{hypdivislem}, finishing the proof.
\end{proof}

The next result gives the first example of a new, non-classical type of vectorial modular form. 

\begin{proposition}
Let $\zeta \in \FF_q^{ac}$, with minimal polynomial $\pfrak$, and let $\Hcal = {h_1 \choose h_2} \in \MM_k^m(\rho_t^*)$. 

The column vector 
\[\left( \ev_\zeta \hyp_t^{(n-1)}h_1,  \ev_\zeta \hyp_t^{(n-2)}h_1,  \dots,   \ev_\zeta h_1 \right)^{tr} \in  \]
\[ M_k^{m+1}(\Gamma_1(\pfrak^n))\times M_k^{m+1}(\Gamma_1(\pfrak^{n-1}))\times \cdots \times M_k^{m+1}(\Gamma_1(\pfrak)) \subset M_k^{m+1}(\Gamma_1(\pfrak^n))^n\]
is a vectorial Drinfeld modular form of weight $k$ and type $m$ for $\Gamma_0(\pfrak^n)$ with representation $\phi_\zeta^{(n)}$. 
\end{proposition}

\begin{proof}
From \eqref{hypertrans}, with $\gamma = \left(\begin{smallmatrix} a & b \\ c & d \end{smallmatrix} \right) \in \Gamma_0(\pfrak^n)$, we obtain
\begin{align*} 
\ev_\zeta(\hyp_{t}^{(n-1)}h_1)(\gamma (z)) = \frac{j(\gamma,z)^k}{(\det\gamma)^{m+1}} \sum_{j = 0}^{n-1} \ev_\zeta (\hyp_{t}^{(j)}d(t)) \ev_\zeta(\hyp_{t}^{(n-1-j)} h_1)(z). \end{align*}
Evaluation at $t = \zeta$ in the previous displayed equation above plus Proposition \ref{hyperdiffprop}.
\end{proof}




\subsection{Regularity at infinity}

As we have seen, the first coordinate of a VMF of weight $k$ rigid analytically interpolates Drinfeld modular forms of weight $k$ for $\Gamma_1(\pfrak)$ for all monic irreducible polynomials $\pfrak \in A$. The next result sharpens this observation, adding justification of our choice of expansion at the infinite cusp.

Recall that $\Upsilon := \left( \begin{smallmatrix} 1 & 0 \\ 0 & u  \end{smallmatrix} \right)$, $\Psi_1 := \left(\begin{smallmatrix} 1 & 0 \\ -\psi_1 & 1 \end{smallmatrix} \right)$, $\Theta_t := \left(\begin{smallmatrix} 1 & 0 \\ -\chi_t & 1 \end{smallmatrix} \right)$, and $\bm{E} := (\mathcal{E}_1, \tau(\mathcal{E}_1))$. There are now several equivalent formulations for condition \eqref{vanishcond}, which we summarize in the following result.

\begin{theorem} \label{reginftythm}
Let $\mathcal{H} = {h_1 \choose h_2} \in \MM_k^m(\Gamma(1), \rho^*_t)^!$. The following are equivalent:

1. $\Hcal \in \MM_k^m(\rho_t^*)$. \medskip

2. $\Upsilon \Theta_t^{-1} \Hcal = \Psi_1^{-1} \Upsilon \Hcal \in u\TT[[u]]^2$. \medskip

3. $\bm{E}^{-1} \Hcal \in (M_{k-1}^m \oplus M_{k-q}^m) \otimes \TT $. \medskip

4. For infinitely many $\zeta \in \FF_q^{ac} \subset \CC_\infty$ with minimal polynomials $\pfrak \in A$, 
\[\ev_\zeta(h_1) \in M_k^m(\Gamma_1(\pfrak)).\]

\end{theorem}

\begin{proof}
Taking into account the Definition \ref{VMFdef}, Corollary \ref{infequiv1}, and the proof of Theorem \ref{structurethm}, only the equivalence between 1. and 4. remains to be shown. 

Suppose $\Hcal = {h_1 \choose h_2} \in \MM_k^m(\Gamma(1),\rho^*_t)$.
Equation \eqref{modularity} shows that 
\[j(\gamma,z)^{-k} \ev_\zeta(h_i)(\gamma (z))\] is a $K(\zeta)$-linear combination of $\ev_\zeta(h_1)(z)$ and $\ev_\zeta(h_2)(z)$, for all $\gamma \in GL_2(A)$.
Thus, it suffices to check the holomorphy at infinity for each coordinate function of $\ev_\zeta\Hcal$. We use \eqref{uexpcond} in the equivalent form given through \eqref{upsiloncommeq}, so that
\[\Hcal = \left(\begin{matrix} u & 0 \\ -\psi_1 & 1 \end{matrix} \right) \left(\begin{matrix} h_1 \\ h_3 \end{matrix} \right), \text{ for some } h_1, h_3 \in \TT[[u]].\]
After Lemma \ref{Psievallem} and the obvious inclusion $\TT[[u]] \subset \TT[[u_\pfrak]]$, the check is clear.

Conversely, suppose that $\ev_\zeta(h_1) \in M_k^m(\Gamma_1(\pfrak))$ for infinitely $\zeta \in \FF_q^{ac}$ with minimal polynomial $\pfrak \in A$.
By the first assumption, $\ev_\zeta(h_1)$ is a modular form with character for $\Gamma_0(\pfrak)$ and by the second it has an expansion in $\CC_\infty[[u_\pfrak]]$ for these primes $\pfrak$.
As we have explained above, $\Gamma_0(\pfrak)$ has only two cusps $0$ and $\infty$, and the matrix $\gamma_0 := \left(\begin{smallmatrix} 0 & -1 \\ 1 & 0 \end{smallmatrix} \right)$ sends $\infty$ to $0$. By \eqref{modularity} we see that $j(\gamma_0,z)^{-k} h_1(\gamma_0 (z)) = -h_2(z)$.
Thus as $j(\gamma_0,z)^{-k} \ev_\zeta(h_1)(\gamma_0 (z)) \in \CC_\infty[[u_\pfrak]]$ (by assumption) for infinitely many primes $\pfrak$, we conclude that the same is true for $\ev_\zeta(h_2)$.

By Lemma \ref{periodicitylem}, as above, we have $\Hcal = \left(\begin{smallmatrix} u & 0 \\ -\psi_1 & 1 \end{smallmatrix} \right) \left(\begin{smallmatrix} h_1 \\ h_3 \end{smallmatrix} \right), \text{ for some } h_1, h_3 \in \TT((u))$, and we must show $h_1,h_3 \in \TT[[u]]$. 
Write 
\[h_1 := \sum_{i \geq \nu} a_i u^i \text{ and } h_3 := \sum_{j \geq \mu} b_i u^j.\]
Focusing on $h_1$, we must have $\nu \geq -1$, since $\ev_\zeta(u h_1)$ is holomorphic at infinity, for infinitely many $\pfrak$, and hence for all $i$ strictly less than $-1$, $a_i$ must vanish at infinitely elements of $\FF_q^{ac}$ which implies that they are zero in $\TT$, in virtue of the fact that $\TT$ is factorial.

With this information, we argue similarly that $\mu \geq -1$  in the expansion of $h_3$ using the holomorphy at infinity of $\ev_\zeta(h_2) = \ev_\zeta(-\psi_1 h_1 + h_3)$.
Now, if $b_{-1} \neq 0$, then we obtain a term of order $u_\pfrak^{-|\pfrak|}$ in $\ev_\zeta(h_2)$,
but since $\ev_\zeta(\psi_1)$ vanishes at infinity to the order of $u_\pfrak^{|\pfrak|(1-q^{-1})}$, the order of the possible pole for $\ev_\zeta (- \psi_1 h_1)$ must be strictly less than $|u_\pfrak|^{-|\pfrak|}$, \emph{i.e.} the maximal order of the pole coming from $\ev(h_1)$. Thus $b_{-1}$ must also vanish, and $h_3 \in \TT[[u]]$.
Finally, by Lemma \ref{Psievallem},
$\ev_\zeta (\psi_1) = \kappa u_\pfrak^{|\pfrak|(1-q^{-1})}+ \sum_{l > |\pfrak|(1-q^{-1})} \kappa_l u_\pfrak^l,$
for some non-zero $\kappa \in \CC_\infty$, and this is not enough to cancel the $u_\pfrak^{-|\pfrak|}$ coming from the assumption that $a_{-1} \neq 0$.
Thus we conclude that $h_1$ is also in $\TT[[u]]$, finishing the proof.
\end{proof}

\section{Hecke operators}\label{Heckeoperators}
\begin{definition}
Let $k,m$ be non-negative integers and $\rho_t^*$ as above. For each $\gamma \in \GL_2(K) \cap \Mat_2(A)$ and each holomorphic function $\mathcal{F} : \Omega \rightarrow \TT^2$, we define a \textit{vectorial slash operator} $[\gamma] = [\gamma]_{k,m,\rho_t^*}$ by
\[\mathcal{F}_{[\gamma]_{k,m,\rho_t^*}} := j(\gamma,z)^{-k} \det(\gamma)^m \rho_t^*(\gamma)^{-1} \mathcal{F}(\gamma (z)).\]
\end{definition}

We state the following easily checked lemma for the record.

\begin{lemma}\label{slashlem}
Let $k,m$ be fixed, and let $[\cdot] = [\cdot]_{k,m,\rho_t^*}$. We have the the following properties of the slash operators: \medskip

\noindent 1. For $\gamma, \gamma' \in \GL_2(K) \cap \Mat_2(A)$, we have $[\gamma \gamma'] = [\gamma][\gamma']$.

\noindent 2. For $\mathcal{F}$ as in the definition above, $\mathcal{F} \in \ \MM_k^m(\rho^*_t)^!$ if and only if $\mathcal{F}_{[\gamma]} = \mathcal{F}$ for all $\gamma \in \Gamma$.
\end{lemma}

To follow $\pfrak$ denotes a monic irreducible polynomial, and $\pfrak A$ is the ideal it generates. Let $\Mat_\pfrak$ be the subset of $\Mat_2(A)$ consisting of those matrices with determinant in $\pfrak\FF_q^\times$. As usual, the group $\Gamma(1)$ acts on $\lquo{\Gamma(1)}{M_\pfrak}$ by right multiplication, permuting the cosets.


\begin{definition}
We define the Hecke operator $T_\pfrak$ on $\MM_k^m(\rho_t^*)$ by
\[T_\pfrak \mathcal{G} := \pfrak^{k-m}\sum_{\beta \in \mbox{\small $\lquo{\Gamma(1)}{M_\pfrak}$}} \mathcal{G}_{[\beta]_{k,m,\rho}}, \text{ for all } \Gcal \in \MM_k^m(\rho_t^*)^!.\]
Notice that here $m$ may be any integer.
\end{definition}

The next result follows immediately from the properties of the slash operators in Lemma \ref{slashlem}.
\begin{lemma}
The operators $T_\pfrak$ do not depend on the choice of representatives for $\lquo{\Gamma(1)}{\Mat_\pfrak}$, and $T_\pfrak$ induces a $\TT$-module endomorphism of $\MM_k^m(\rho_t^*)^!$. \hfill $\qed$
\end{lemma}

For each non-negative integer $d$, let $A(d)$ be the $\FF_q$-vector subspace of $A$ of polynomials whose degree is strictly less than $d$. The following lemma is readily checked using elementary techniques.

\begin{lemma}
The matrices $\left(\begin{smallmatrix} 1 & b \\ 0 & \pfrak \end{smallmatrix} \right), \left(\begin{smallmatrix} \pfrak & 0 \\ 0 & 1 \end{smallmatrix} \right)$, for $b \in A(\deg \pfrak)$, give a full set of representatives for the coset space $\lquo{\Gamma(1)}{\Mat_\pfrak}$.
\end{lemma}

\begin{corollary}
For all $\mathcal{F} \in \MM_k^m(\rho_t^*)^!$, we have
\[ (T_{\pfrak} \mathcal{F})(z) :=  \pfrak^k \left(\begin{matrix} \chi_t(\pfrak) & 0 \\ 0 & 1 \end{matrix} \right) \mathcal{F}(\pfrak z) + \sum_{b \in A(\deg \pfrak)} \left( \begin{matrix} 1 & 0 \\ \chi_t(b) & \chi_t(\pfrak) \end{matrix} \right) \mathcal{F}\left(\frac{z+b}{\pfrak} \right).\]
\end{corollary}

The explicit description above gives the following immediate relation with the Anderson twist $\tau$.

\begin{corollary} \label{Hecketaucompat}
For all $\pfrak$, we have $\tau T_\pfrak = T_\pfrak \tau$. 
\end{corollary}

The following result, standard in the theory of Drinfeld modular forms, follows, and the verification is left to the reader.

\begin{corollary}
The Hecke operators are totally multiplicative.
\end{corollary}

We record here the following description of the action on the coordinates of the Hecke operators which will be useful to follow.

\begin{corollary} \label{heckcoordesc}
Let $\mathcal{H} = {h_1 \choose h_2} \in  \MM_k^m(\rho_t^*)^!$. One has
\[T_\pfrak{h_1 \choose h_2}(z) = \left( \begin{array}{c} \pfrak^k \chi_t(\pfrak) h_1(\pfrak z) + \sum h_1\left(\frac{z+b}{\pfrak}\right) \\ \pfrak^k h_2(\pfrak z) + \chi_t(\pfrak) \sum h_2\left(\frac{z+b}{\pfrak}\right) \end{array} \right) + \left( \begin{array}{c} 0 \\ \sum \chi_t(b) h_1 \left(\frac{z+b}{\pfrak}\right) \end{array} \right) .\]
Here each sum $\sum$ is over ${b \in A(\deg \pfrak)}$.
\end{corollary}

\begin{remark}
Upon specialization of $t$ at a root $\zeta$ of the monic irreducible polynomial $\pfrak$, the function $\ev_\zeta (\sum \chi_t(b) h_1 \left(\frac{z+b}{\pfrak}\right))$ is closely related to one of the twisted Hecke operators from \cite{RPmf16} applied to $\ev_\zeta(h_1)$.
\end{remark}

\subsection{Hecke operators preserve regularity at infinity.}\label{regularityhecke}
We have argued above that the Hecke operators stabilize the $\TT$-modules $\MM_k^m(\rho^*_t)^!$, for all $k$ and $m \pmod{q-1}$. It remains to show that they preserve the regularity at infinity of Theorem \ref{reginftythm}. We require a preliminary result on $\chi_t$ and some further notation.

For all $d \geq 0$, define
\[E_d(z) := D_d^{-1} \prod_{a \in A(d)}(z-a) \ \text{ and } b_d(t) := \prod_{j = 0}^{d-1} (t - \theta^{q^j}),\]
with $b_0(t) = 1$.

 \begin{lemma} \label{AGFfe}
For all $a \in A$ of degree $d$, 
\[\chi_t(az) = \chi_t(a)\chi_t(z) + \omega^{-1}\sum_{l = 0}^{d-1} \sum_{l < i \leq d} E_i(a) \tau^{i - l} (b_l) \ec(z)^{q^{i - (l+1)}}.\]
In particular, $u(z)^{|a|}\chi_t(az) \rightarrow 0$ as $|z|_\Im \rightarrow \infty$.
\begin{proof}
Recall $f_t$ from \eqref{AGFdef} and its $\tau$-difference equation which may be deduced from \eqref{taudifferencechi}. One readily computes by induction that for all non-negative integers $n$,
\[\tau^n(f_t) = b_n f_t  + \sum_{l = 0}^{n-1} \tau^{n-l}(b_l) \tau^{n-(l+1)}(\ec).\]
Then using the fact that $f_t(az) = \mathfrak{c}_a(f_t(z))$, we obtain the first identity.

When $d = \deg a$ equals $0$, the second claim follows from Lem. \ref{chipropslem}, and then it is also clear for all $d \geq 1$.
\end{proof}
\end{lemma}

\begin{proposition} \label{Heckestable}
The operator $T_\pfrak$ stabilizes both $\MM_k^m(\rho^*_t)$ and $\SSS_k^m(\rho^*_t)$.
\end{proposition}

\begin{proof}
We use both 1. and 2. from Theorem \ref{reginftythm}. By 1. of the aforementioned theorem, we may write 
\[\Hcal = \Theta_t{h_1 \choose h_2} = {h_1 \choose h_2 - \chi_t h_1}, \text{ with both } h_1, u h_2 \in u\TT[[u]].\]
From Cor. \ref{heckcoordesc} we obtain,
\begin{flalign*}
T_\pfrak \Hcal =  \left( \begin{smallmatrix} \pfrak^k \chi_t(\pfrak) h_1(\pfrak z) + \sum h_1\left(\frac{z+b}{\pfrak}\right) \\ \pfrak^k (h_2 - \chi_t h_1)(\pfrak z) + \chi_t(\pfrak) \sum (h_2  - \chi_t h_1)\left(\frac{z+b}{\pfrak}\right)  \end{smallmatrix} \right) 
  + \left( \begin{smallmatrix} 0 \\ \sum \chi_t(b) h_1 \left(\frac{z+b}{\pfrak}\right) \end{smallmatrix} \right),
\end{flalign*}
where again each sum $\sum$ is over ${b \in A(d)}$.

Abbreviate  
\[T_\pfrak^\star \Hcal := \left( \begin{smallmatrix} \pfrak^k \chi_t(\pfrak) h_1(\pfrak z) + \sum h_1\left(\frac{z+b}{\pfrak}\right) \\ \pfrak^k h_2(\pfrak z)  + \chi_t(\pfrak) \sum h_2\left(\frac{z+b}{\pfrak}\right)  \end{smallmatrix} \right), \text{ and}\]

$\upsilon := -\pfrak^k \chi_t(\pfrak z) h_1(\pfrak z) - \chi_t(\pfrak) \sum \chi_t\left(\frac{z+b}{\pfrak}\right)h_1\left(\frac{z+b}{\pfrak}\right) + \sum \chi_t(b) h_1 \left(\frac{z+b}{\pfrak}\right).$ \newline
Then 
\[T_\pfrak\Hcal = T_\pfrak^\star\Hcal + {0 \choose \upsilon}.\]
One checks as in \cite[(7.3)]{EGinv} that $\Upsilon T_\pfrak^\star \Hcal \in u\TT[[u]]^2$.
If one additionally assumes that $h_2 \in u\TT[[u]]$, then $T_\pfrak^\star \Hcal \in u\TT[[u]]^2$.

Now we focus on $\upsilon$.
Already by Lem. \ref{AGFfe}, we see that $\pfrak^k \chi_t(\pfrak z) h_1(\pfrak z)$ vanishes at infinity, since $u(z)^{|\pfrak|}$ divides $h_1(\pfrak z)$ in $\TT[[u(z)]]$. 
By the same lemma and letting $\deg(\pfrak) = d$, we have equality between $\chi_t(\pfrak) \chi_t\left( \frac{z+b}{\pfrak} \right)$ and
\[\chi_t(z)+\chi_t(b) - \omega^{-1}\sum_{l = 0}^{d-1} \sum_{l < i \leq d} E_i(a) \tau^{i - l} (b_l) \ec\left( \frac{z+b}{\pfrak}  \right)^{q^{i - (l+1)}}.\]
Hence, letting 
\begin{align} \label{r0def} 
r_0 := \omega^{-1} \sum_{l = 0}^{d-1} \sum_{l < i \leq d} E_i(a) \tau^{i - l} (b_l) \sum_{b \in A(d)} h_1\left(\frac{z+b}{\pfrak}\right) \ec\left( \frac{z+b}{\pfrak}  \right)^{q^{i - (l+1)}}\end{align}
we see \medskip

$-\chi_t(\pfrak) \sum \chi_t\left(\frac{z+b}{\pfrak}\right) h_1\left(\frac{z+b}{\pfrak}\right) + \sum \chi_t(b) h_1 \left(\frac{z+b}{\pfrak}\right) = r_0 -\chi_t(z)\sum h_1\left(\frac{z+b}{\pfrak}\right),$ \medskip 

\noindent and by \cite[(7.3)]{EGinv} 
\[\chi_t(z)\sum h_1\left(\frac{z+b}{\pfrak}\right) \in u\chi_t \TT[[u]].\]

Focusing on the sum over $A(d)$ in $r_0$, we write $h_1 = \sum_{j \geq 1} \eta_j u^j$ in its $u$-expansion, using our assumption that $h_1 \in u \TT[[u]]$. Fixing $0 \leq m_0 \leq d-1$ and $m_0 < i_0 \leq d$ and letting $\mu := q^{i_0 - (m_0+1)} \leq |\pfrak|/q$, the sum over $A(d)$ becomes
\begin{flalign*}
\sum_{b \in A(d)} h_1 \left( \frac{z+b}{\pfrak} \right) \ec\left( \frac{z+b}{\pfrak} \right)^{\mu} =& \sum_{j \geq 1} \eta_j \sum_{b \in A(d)} u\left( \frac{z+b}{\pfrak} \right)^{j - \mu} \\
 = & \sum_{1 \leq j \leq \mu} \eta_j \sum_{b \in A(d)} \ec\left( \frac{z+b}{\pfrak} \right)^{\mu - j} \\
 & + \sum_{j > \mu} \eta_j \sum_{b \in A(d)} u\left( \frac{z+b}{\pfrak} \right)^{j - \mu} \\
 & =: U + V.
\end{flalign*}
Gekeler has shown that for all $j \geq 1$, we have both 
\[\sum_{b \in A(d)} u\left( \frac{z+b}{\pfrak} \right)^j = G_{\pfrak,j}(\pfrak u(z)),\] 
where $G_{\pfrak,j}$ is the Goss polynomial for the lattice of Carlitz $\pfrak$-torsion $C[\pfrak]$, see \cite[(7.3)]{EGinv}, and $u$ divides $G_{\pfrak,j}(\pfrak u)$, see \cite[(3.9)]{EGinv}.
Thus $V$ is in $u\TT[[u]]$.
One sees that $U$ vanishes by use of the linearity of the Carlitz exponential function, the binomial theorem and the vanishing of the sums $\sum_{\lambda \in C[\pfrak]} \lambda^l$ for $0 \leq l < |\pfrak|-1$ while noting again that for all choices of $i_0$ and $m_0$ as above $\mu \leq |\pfrak|/q$.
Thus we conclude that $r_0 \in u\TT[[u]]$, finishing the proof.
\end{proof}

\begin{remark}[Necessity of vanishing of first coordinate for $T_\pfrak$-stability]
Closer inspection of the proof demonstrates that the assumption that $h_1 \in u\TT[[u]]$ in the expansion $\Hcal = \Theta_t {h_1 \choose h_2}$ plays a crucial role in showing that the $\pfrak^k \chi_t(\pfrak z) h_1(\pfrak z)$ term appearing in the second coordinate vanishes at infinity after multiplication by $u$. Indeed, if we only assume that $h_1 \in \TT[[u]]$, by Lem. \ref{AGFfe} we see that $\pfrak^k \chi_t(\pfrak z) h_1(\pfrak z)$ grows like $u^{-|\pfrak|/q}$, and thus for all monic irreducibles $\pfrak$, $u(z)\pfrak^k \chi_t(\pfrak z) h_1(\pfrak z)$ does not tend to zero as $|z|_\Im \rightarrow \infty$.

In \S \ref{nonexample}, we have seen that $E^{(k)}\Fcal$ is not regular at infinity for any $k \geq 1$, and further $(uE^{(k)}\Fcal)(z)$ does not tend to zero as $|z|_\Im \rightarrow \infty$, which is a necessary condition for all VMF. In particular, for all such one dimensional Eisenstein series $E^{(k)}$, the vectorial function $u T_\pfrak (E^{(k)}\Fcal^*)$ does not tend to zero as $|z|_\Im \rightarrow \infty$. 
\end{remark}

\begin{remark}
We examine the effect of the Hecke operators on the expansion of a VMF $\Hcal = \Theta_t {h_1 \choose h_2}$ in $\Theta_t\TT[[u]]$. A closer inspection of the proof of Prop. \ref{Heckestable} shows that we may write
\[ T_\pfrak \left( \begin{smallmatrix} h_1 \\ h_2 - \chi_t h_1 \end{smallmatrix} \right)(z) = \left( \begin{matrix} 1 & 0 \\ -\chi_t(z) & 1 \end{matrix} \right) \left( \begin{smallmatrix} \pfrak^k \chi_t(\pfrak) h_1(\pfrak z) + \sum h_1\left(\frac{z+b}{\pfrak}\right) \\ \pfrak^k h_2(\pfrak z)  + \chi_t(\pfrak) \sum h_2\left(\frac{z+b}{\pfrak}\right) + r_0 + r_1 \end{smallmatrix} \right) \in \Theta_t\TT[[u]], \]
where $r_0 \in u\TT[[u]]$ is as defined in \eqref{r0def}, and 
\[r_1 := \frac{\pfrak^k}{\omega} h_1(\pfrak z) \sum_{m = 0}^{d-1} \sum_{m < i \leq d} E_i(\pfrak) \tau^{i-m}(b_m) \mathfrak{e_c}(z)^{q^{i-(m+1)}} \in u^{\frac{|\pfrak|(q-1)}{q}}\TT[[u]]\]
comes from $-\pfrak^k \chi_t(\pfrak z) h_1(\pfrak z)$ using Lemma \ref{AGFfe}.
\end{remark}

\subsection{Hecke eigenforms: First examples}\label{eigenformsFirstexamples}
\subsubsection{Weight one forms}
The $\TT$-module $\MM_1^0(\rho_t^*)$ has rank one, and hence the Eisenstein series $\Ecal_1$ is a Hecke eigenform for all $\pfrak$. We compute that $T_\pfrak \Ecal_1 = \pfrak \Ecal_1$ below. 

As a corollary of the $A$-expansion of $\Ecal_1$ given above, we see that through the evaluations $t \mapsto \theta^{q^k}$, for $k \geq 1$, this single VMF gives rise to infinitely many Drinfeld $T_\pfrak$-eigenforms for $GL_2(A)$ with eigenvalue $\pfrak$. These specialized forms are nothing more than than Petrov's forms $f_s$ with $s = \frac{q^k - 1}{q-1}$.

\subsubsection{Weight $q$ forms}
The $\TT$-module $\MM_q^0(\rho_t^*)$ has rank two with generators the non-cuspidal Eisenstein series $\Ecal_q$ and the cuspidal $h\Fcal^* := h {-\bm{d}_2 \choose \bm{d}_1}$. We show below that for all $\pfrak$, we have $T_\pfrak \Ecal_q = \pfrak^q \Ecal_q$. 

Now we focus on $-h\Fcal^*$. Since $\SSS_q^0(\rho_t^*)$ has rank one, it suffices to determine the coefficient of $u$ in the first coordinate of $T_\pfrak(h \Fcal^*)$. Writing $\bm{d}_2 = 1 + \sum_{j \geq q-1} c_j u^j \in A[t][[u]]$, as in \eqref{eqd2}, and acting by $T_\pfrak$, the first coordinate becomes
\[\pfrak^q  h(\pfrak z) (\chi_t(\pfrak) - \pfrak) + \pfrak^{q+1} h(\pfrak z) + \sum h\left(\frac{z+b}{\pfrak}\right) +  o(u^2),\] by the same reasoning as in \cite[(7.6)]{EGinv} because $u^q$ divides $ h \sum_{j \geq q-1}c_j u^j$ and where we have added and subtracted $\pfrak^{q+1} h(\pfrak z)$. Now $u^{|\pfrak|}$ divides $h(\pfrak z)$, and $\pfrak^{q+1} h(\pfrak z) + \sum h\left(\frac{z+b}{\pfrak}\right) = \pfrak h(z)$, by \emph{ibid}. Thus we conclude that $T_\pfrak(h \Fcal^*) = \pfrak h\Fcal^*$. 

\begin{remark}
Here again we see, as in the classical Drinfeldian setting over $\CC_\infty$, that two forms of different weights may have the same Hecke eigenvalues, as we have just observed for $\Ecal_1$ and $h \Fcal^*$. Perhaps the fact that both of these vectors appear as columns of the inverse of the rigid analytic trivialization matrix $\Psi := \left( \begin{smallmatrix} \bm{d}_1 & \bm{d}_2 \\ \tau(\bm{d}_1) & \tau(\bm{d}_2) \end{smallmatrix} \right)$ may have some explanatory power. 
\end{remark}

\subsubsection{Remark: Weight $q+2$ forms}
Finally, we consider the form $h\Ecal_1$, which lies in the rank one $\TT$-module $\SSS_{q+2}^1(\rho_t^*)$. Necessarily, this is a Hecke eigenform, and A. Petrov has computed a handful of examples which suggest that $T_\pfrak (h\Ecal_1) = \pfrak^2 h \Ecal_1$. A proof would appear to take somewhat more work than required in the previous subsections, essentially due to the double cuspidality of the first coordinate and the more complicated behavior of the Goss polynomials for Carlitz torsion lattices in this circumstance. 
We do not pursue this here.

\subsection{Hecke properties of the vectorial Eisenstein series}
\begin{proposition} \label{EisHeckEFs}
For all monic irreducibles $\pfrak \in A$ and $k \equiv 1 \pmod{q-1}$, we have
\[T_\pfrak \Ecal_k = \pfrak^k \Ecal_k.\]
\end{proposition}

The proof that the $\Ecal_k$ are Hecke eigenforms takes up the next several subsections. We use heavily the coordinate description of the action of $T_\pfrak$ from Cor. \ref{heckcoordesc}.

\subsubsection{The first coordinate}
The next result is well-known, see e.g. it is implicit in the calculation of the Hecke eigenvalues of the Eisenstein series studied in \cite{GossCM}. 

\begin{lemma} \label{Gorthlem}
For all $a \in A_+$ and positive integers $k$,
\[\sum_{b \in A(d)} G_k \left(u_a \left( \frac{z+b}{\pfrak} \right) \right) = \left\{\begin{array}{ll} \pfrak^k G_k(u_a(z)) & (a,\pfrak) = 1, \\ 0 & (a,\pfrak) \neq 1. \end{array} \right. \]
\hfill \qed \end{lemma}

From Cor. \ref{heckcoordesc} we must compute 
\[(T_\pfrak^\sharp e_1^k)(z) :=  \pfrak^k \chi_t(\pfrak) e_1^k(\pfrak z) + \sum_{b \in A(d)} e_1^k \left(\frac{z+b}{\pfrak} \right).\] 
Writing everything out and using the previous lemma, we have
\[ (T_\pfrak^\sharp e_1^k)(z) = \pfrak^k \sum_{a \in A_+} \chi_t(\pfrak a) G_k(u_{\pfrak a}(z)) + \sum_{a \in A_+} \chi_t(a) \sum_{b \in A(d)} G_k \left(u_a \left( \frac{z+b}{\pfrak} \right) \right) \]
\[ = \pfrak^k \sum_{a \in A_+} \chi_t(\pfrak a) G_k(u_{\pfrak a}(z)) + \pfrak^k \sum_{a \in A_+ \setminus \pfrak A_+} \chi_t(a) \sum_{b \in A(d)} G_k \left(u_a \left( z \right) \right) = \pfrak^k e_1^k(z), \]
finishing the calculation. 

\subsubsection{The second coordinate}
\begin{lemma}
Let $a \in A_+$, $\pfrak$ a monic irreducible of positive degree $d$, and $k$ a positive integer $k \equiv 1 (q-1)$. We have that $\sum_{c \in A(d)} \sum_{b \in A} \frac{\chi_t(\pfrak b)}{(az+ac+\pfrak b)^k}$ equals 
\[ \sum_{b \in A}  \frac{\chi_t(b)}{(az+b)^k}  - \frac{\chi_t(a)}{\pfrak} \sum_{c \in A(d)} \chi_t(c) G_k ( u_a \left( \frac{z+c}{\pfrak} \right) ), \text{ if } (a,\pfrak) =  1,\] 
and equals zero otherwise.


\end{lemma}
\begin{proof}
Suppose first that $(a,\pfrak) \neq 1$ and write $a' = a/\pfrak \in A_+$. We begin with the case $k = 1$. We have
\[ \sum_{c \in A(d)} \sum_{b \in A} \frac{\chi_t(\pfrak b)}{(az+ac+\pfrak b)} = \frac{\chi_t(\pfrak)}{\pfrak}\sum_{c \in A(d)} \sum_{b \in A} \frac{\chi_t(b)}{(a'(z+c)+b)} \]
\begin{eqnarray*}
&=& \frac{\chi_t(\pfrak)}{\pfrak}\sum_{c \in A(d)} \pitilde u(a'(z+c)) \chi_t(a'(z+c)) \\
&=& \frac{\chi_t(\pfrak)}{\pfrak}\pitilde u(a' z)\left( \chi_t(a' z) \sum_{c \in A(d)} 1 + \chi_t(a')\sum_{c \in A(d)} \chi_t(c)\right) \\
&=& 0, 
\end{eqnarray*}
by the well-known vanishing of power sums. We notice that $\sum_{c \in A(d)} \sum_{b \in A} \frac{\chi_t(b)}{(a'(z+c)+b)}$ and $\sum_{c \in A(d)} \sum_{b \in A} \frac{\chi_t(b)}{(a'(z+c)+b)^{k}}$ are connected via simple hyperdifferentiation with respect to $z$, and thus, when $\pfrak | a$, we learn
$\sum_{b \in A} \sum_{c \in A(d)} \frac{\chi_t(\pfrak b)}{(az+ac+\pfrak b)^k} = 0$.


Now suppose $(a,\pfrak) = 1$, then the sum of our interest $\sum_{c \in A(d)} \sum_{b \in A} \frac{\chi_t(\pfrak b)}{(az+ac+\pfrak b)^k}$ equals
\[\sum_{c \in A(d)} \sum_{b \in A} \frac{\chi_t(ac + \pfrak b)}{(az+ac+\pfrak b)^k} -  \sum_{c \in A(d)} \sum_{b \in A} \frac{\chi_t(ac)}{(az+ac+\pfrak b)^k} \]
\[= \sum_{b' \in A} \frac{\chi_t(b')}{(az + b')^k} - \frac{\chi_t(a)}{\pfrak^{k}} \sum_{c \in A(d)} \chi_t(c) G_k(u_a((z+c)/\pfrak)) . \]
\end{proof}

The following result should be compared with Lemma \ref{Gorthlem}. 
\begin{corollary}
With the same notations,
\[\sum_{c \in A(d)} \chi_t(c) G_k \left(u_a \left( \frac{z+c}{\pfrak} \right) \right) = \left\{ \begin{array}{ll}
 \begin{array}{l} \frac{\pfrak^k}{\chi_t(a)} \sum_{b \in A}  \left( \frac{\chi_t(b)}{(az+b)^k} \right. \\ \left. -\sum_{c \in A(d)}  \frac{\chi_t(\pfrak b)}{(az+ac+\pfrak b)^k} \right) \end{array} & \text{ if } (a,\pfrak) =  1 \\ 0 & \text{ if } (a,\pfrak) \neq 1. \end{array} \right. \] 
\end{corollary}
\begin{proof}
This is a simple restatement of the previous lemma.
\end{proof}

Given the previous lemmas, we continue with the calculation of the Hecke action on the second coordinate of the Eisenstein series $\mathcal{E}_k$. To show that the second coordinate is fixed by the Hecke action defined above, it is equivalent to prove
\begin{equation} \label{equivheckepfeq}
\pfrak^k e_2^k(\pfrak z) + \chi_t(\pfrak) \sum_{c \in A(d)} e_{2}^{k}((z+c)/\pfrak) = \pfrak^k e_2^k(z) - \sum_{c \in A(d)} \chi_t(c) e_1^k((z+c)/b).
\end{equation}

 We have, for $k \equiv 1 (q-1)$, 
\[-e_2^{k}(z) = -\sideset{}{'}\sum_{a,b \in A} \frac{\chi(b)}{(az + b)^k} = L(\chi_t,k) + \sum_{a \in A_+} \sum_{b \in A} \frac{\chi_t(b)}{(az+b)^{k}}.\]
Now we compute using the previous lemma that,
\[ \pfrak(t) \sum_{c \in A(d)} -e_2^{k}((z+c)/\pfrak) = \pfrak^k \sum_{a \in A_+} \sum_{b \in A} \sum_{c \in A(d)}  \frac{\chi_t(\pfrak b)}{(az+ac+\pfrak b)^k} \]
\begin{eqnarray*}
&=& \pfrak^k \sum_{a \in A_+ \setminus \pfrak A_+} \left( \sum_{b \in A} \frac{\chi_t(b)}{(az+b)^k} - \frac{\chi_t(a)}{\pfrak^k} \sum_{c \in A(d)} \chi_t(c) G_k((u_a(z+c)/\pfrak)) \right) \\
 &=& \pfrak^k \sum_{a \in A_+ \setminus \pfrak A_+} \sum_{b \in A} \frac{\chi_t(b)}{(az+b)^k} - \sum_{c \in A(d)} \chi_t(c) \sum_{a \in A_+} \chi_t(a) G_k((u_a(z+c)/\pfrak)) \\
 &=& \pfrak^k \sum_{a \in A_+ \setminus \pfrak A_+} \sum_{b \in A} \frac{\chi_t(b)}{(az+b)^k} + \sum_{c \in A(d)} \chi_t(c) e_1^k((z+c)/\pfrak). 
\end{eqnarray*}

Putting everything together gives \eqref{equivheckepfeq}, and completes the proof of Proposition \ref{EisHeckEFs}. \hfill $\qed$

\subsection{Hecke compatibility: evaluation and hyperdifferentiation}
Now we show that the Hecke operators introduced for VMF above specialize at roots of unity in $\FF_q^{ac} \subset \CC_\infty$ to the Hecke operators for $\Gamma_1(\pfrak^n)$ and commute with hyperdifferentiation; the main result is Theorem \ref{heckespeccommute} below.

\subsubsection{Hecke operators on $\Gamma_1(\mfrak)$}
The following lemma allows us to define Hecke operators on the spaces $M_k^m(\Gamma_1(\mfrak))$. Throughout $\pfrak$ will denote a monic irreducible polynomial in $A$. The following fact is elementary, and we omit the proof.

\begin{lemma}
Let $\mfrak, \pfrak \in A_+$, with $\pfrak$ additionally irreducible. 

If $(\mfrak,\pfrak) = 1$, then the matrices $\left( \begin{smallmatrix} 1 & \beta \\ 0 & \pfrak \end{smallmatrix} \right)$, with $|\beta|< |\pfrak|$, and any matrix $\left( \begin{smallmatrix} \mu & \nu \\ \mfrak & \pfrak \end{smallmatrix} \right)\left( \begin{smallmatrix} \pfrak & 0 \\ 0 & 1 \end{smallmatrix} \right) \in \Mat_2(A)$ such that $\mu \pfrak - \nu \mfrak = 1$ give a full set of distinct representatives for the quotient $\lquo{\Gamma_1(\mfrak)}{\Gamma_1(\mfrak)\left( \begin{smallmatrix} 1 & 0 \\ 0 & \pfrak \end{smallmatrix} \right)\Gamma_1(\mfrak)}$.

If $\pfrak | \mfrak$, the matrices $\left( \begin{smallmatrix} 1 & \beta \\ 0 & \pfrak \end{smallmatrix} \right)$, with $|\beta|< |\pfrak|$ give a full set of distinct representatives. \hfill \qed
\end{lemma}

\begin{definition}
For $f \in M_k^m(\Gamma_1(\mfrak))$ and a monic irreducible $\pfrak \in A$, we define
\[T_\pfrak f := \left\{ \begin{array}{ll} \pfrak^{k-m}\left(\sum_{|\beta|<|\pfrak|} f|_k^m[\left( \begin{smallmatrix} 1 & \beta \\ 0 & \pfrak \end{smallmatrix} \right)] + f|_k^m[\left( \begin{smallmatrix} \mu\pfrak & \nu \\ \mfrak\pfrak & \pfrak \end{smallmatrix} \right)] \right), & (\mfrak, \pfrak) = 1 \\ \pfrak^{k-m}\sum_{|\beta|<|\pfrak|} f|_k^m[\left( \begin{smallmatrix} 1 & \beta \\ 0 & \pfrak \end{smallmatrix} \right)], & (\mfrak,\pfrak) > 1, \end{array} \right. \]
where if $(\mfrak,\pfrak) = 1$ we take $\mu,\nu \in A$ such that $\mu\pfrak - \nu\mfrak = 1$.
\end{definition}

\subsubsection{Compatibility results}
We remind the reader that by Prop. \ref{hyperdiffprop} if $\zeta \in \FF_q^{ac}$ is a root of the monic irreducible $\qfrak \in A$ and $\Hcal \in \MM_k^{m-1}(\rho_t^*)$, then, for all positive integers $n$, 
\[\ev_\zeta(\hyp_{t}^{(n-1)}[\Hcal]_1) \in M_k^{m}(\Gamma_1(\qfrak^n)).\]

\begin{theorem} \label{heckespeccommute}
Let $\Hcal \in \MM_k^{m-1}(\rho_t^*)$, and let $\qfrak \in A_+$ be an irreducible with root $\zeta \in \FF_q^{ac}$. If $(\qfrak, \pfrak) = 1$, then, for all $n \geq 1$, we have
\[\ev_\zeta(\hyp_{t}^{(n-1)}[T_\pfrak \Hcal]_1) = T_\pfrak \ev_\zeta(\hyp_{t}^{(n-1)}[\Hcal]_1) .\]

If $\qfrak = \pfrak$, then, for all $n \geq 1$, we have
\begin{align*}
\ev_\zeta(\hyp_{t}^{(n-1)}[T_\pfrak \Hcal]_1)(z) - T_\pfrak \ev_\zeta(\hyp_{t}^{(n-1)}[\Hcal]_1)(z) = \\
= \pfrak^k\sum_{j = 1}^{n-1} \ev_\zeta((\hyp_{t}^{(j)} \pfrak(t)) (\hyp_{t}^{(n-1-j)} [\Hcal]_1))(\pfrak z).\end{align*}
Hence, $\ev_\zeta(\hyp_{t}^{(n-1)}[T_\pfrak \Hcal]_1) - T_\pfrak \ev_\zeta(\hyp_{t}^{(n-1)}[\Hcal]_1)$ comes from modular forms of lower level. 
\end{theorem}

\begin{proof}
First suppose $(\qfrak,\pfrak) = 1$. From the definition of the Hecke operators on $M_k^{m}(\Gamma_1(\qfrak^n))$, we have
\begin{flalign*}
T_\pfrak \ev_\zeta(\hyp_{t}^{(n-1)}[\Hcal]_1)(z) &- \pfrak^{k-m} \sum_{|\beta|<|\pfrak|} \ev_\zeta(\hyp_{t}^{(n-1)}[\Hcal])|_k^m[\left( \begin{smallmatrix} 1 & \beta \\ 0 & \pfrak \end{smallmatrix} \right)](z) =  \\
&= \pfrak^{k-m}\ev_\zeta(\hyp_{t}^{(n-1)}[\Hcal])|_k^m[\left( \begin{smallmatrix} \mu & \nu \\ \qfrak^n & \pfrak \end{smallmatrix} \right)\left( \begin{smallmatrix} \pfrak & 0 \\  0 & 1 \end{smallmatrix} \right)](z) \\
&= \pfrak^k\sum_{j = 0}^{n-1} \ev_\zeta(\hyp_{t}^{(j)}\pfrak(t)\hyp_{t}^{(n-1-j)}[\Hcal]_1)(\pfrak z), \end{flalign*}
where the last equality follows from \eqref{hypertrans}, since $\left( \begin{smallmatrix} \mu & \nu \\ \qfrak^n & \pfrak \end{smallmatrix} \right) \in \Gamma_0(\qfrak^n)$.

From Cor. \ref{heckcoordesc} and the Leibniz rule, we have
$\hyp_{t}^{(n-1)}[T_\pfrak\Hcal]_1(z)$ equals 
\[\pfrak^k\sum_{j = 0}^{n-1} (\hyp_{t}^{(j)} \pfrak(t)) (\hyp_{t}^{(n-1-j)}[\Hcal]_1)(\pfrak z) + \sum_{|\beta| < |\pfrak| } (\hyp_{t}^{(n-1)} [\Hcal]_1) \left( \frac{z+\beta}{\pfrak} \right). \]
Evaluating at $t = \zeta$ finishes the proof when $(\qfrak, \pfrak) = 1$. 

When $\qfrak = \pfrak$, the claim follows from the previous line. 
\end{proof}

The following corollaries are immediate. 

\begin{corollary}
Let $\Hcal \in \MM_k^{m-1}(\rho_t^*)$ and let $\zeta \in \FF_q^{ac}$, with minimal polynomial $\qfrak$. Let $n \geq 1$. 

For all monic irreducibles $\pfrak \in A$ different from $\qfrak$, if $\Hcal$ is a Hecke eigenform for $T_\pfrak$ with eigenvalue $\lambda_\pfrak \in \CC_\infty$, then $\ev_\zeta(\hyp_{t}^{(n-1)}[\Hcal]_1)$ is a Hecke eigenform for $T_\pfrak$, with the same eigenvalue.

If $\pfrak = \qfrak$, the same statement holds verbatim when $n = 1$, and it holds modulo the space of oldforms when $n > 1$. \hfill $\qed$
\end{corollary}

\begin{corollary} \label{HEFHLcor}
For all positive $k \equiv 1 \pmod{q-1}$, all $n \geq 1$, and all $\zeta \in \FF_q^{ac}$ with minimal polynomial $\qfrak \in A_+$, the forms $\ev_\zeta(\hyp_{t}^{(n-1)}[\Ecal_k]_1) \in M_k^1(\Gamma_1(\qfrak^{n}))$ are simultaneous Hecke eigenforms for the family $\{T_\pfrak : \pfrak(\zeta) \neq 0\}$ with eigenvalues $\{\pfrak^k\}$. 
\end{corollary}

\begin{remark}
We point out a couple of glaring differences in the behavior of the coefficients for Drinfeld Hecke eigenforms with $A$-expansions for $\Gamma_1(\qfrak^n)$ and for classical Hecke eigenforms. 

Consider the form $\ev_\zeta(\hyp_{t}^{(n-1)}[\Ecal_1]_1) \in M_k^1(\Gamma_1(\qfrak^{n}))$ which is a Hecke eigenform with $A$-expansion
\[ \ev_\zeta(\hyp_{t}^{(n-1)}[\Ecal_k]_1) =  -\pitilde\sum_{a \in A_+} \ev_\zeta\hyp^{(n-1)}_t(a(t)) u(az). \]
Firstly, the $A$-expansion coefficient of $u(z)$, which is analogous to the coefficient of $q^1$ in the classical case, can be zero here. Indeed, as soon as $n \geq 2$, this $A$-expansion coefficient vanishes. 
So one cannot hope for the coefficient of $u(\pfrak z)$, analogous to the coefficient of $q^p$ in the classical case, of a Hecke eigenform to come from the Hecke eigenvalue of $T_\pfrak$ and the coefficient of $u(z)$ in the same way as happens classically. 
Second, the Hecke eigenvalues of an eigenform with $A$-expansion are completely determined by which Goss polynomial $G_k$ appears in the $A$-expansion for the form --- this follows from a modified version of Petrov's \cite[Theorem 2.3]{PAjnt} and is essentially Lemma \ref{Gorthlem} above; 
for the form above $\ev_\zeta(\hyp_{t}^{(n-1)}[\Ecal_k]_1)$, $G_1(X) = X$ appears and forces the eigenvalue of $T_\pfrak$ to be $\pfrak$. Still, for each prime $\pfrak$ away from the level $\qfrak$, the coefficient of $u(\pfrak z)$, namely $ \ev_\zeta\hyp^{(n-1)}_t(\pfrak(t))$, is a function of the eigenvalue $\pfrak$ of $T_\pfrak$ and the level. 
\end{remark}

\end{document}